\documentclass[a4paper,12pt,reqno]{amsart}
\usepackage{amssymb}
\usepackage{latexsym}
\usepackage{amsfonts}
\usepackage{mathrsfs}
\usepackage{amsmath}
\usepackage{amsthm}
\usepackage{multicol}
\usepackage{bm}
\usepackage{constants}
\theoremstyle{definition}
\newtheorem{definition}{Definition}[section]
\newtheorem{proposition}[definition]{Proposition}
\newtheorem{theorem}[definition]{Theorem}

\newtheorem{lemma}[definition]{Lemma}

\newtheorem{remark}[definition]{Remark}

\def\div{\mbox{div}\,}

\def\<{\mathop{<}}
\def\>{\mathop{>}}

\newcommand{\spt}{\mathrm{spt}\,}

\newcommand{\diam}{\mathrm{diam}\,}
\numberwithin{equation}{section}

\def\XXint#1#2#3{{\setbox0=\hbox{$#1{#2#3}{\int}$}
\vcenter{\hbox{$#2#3$}}\kern-.5\wd0}}

\title[Existence of MCF with transport term and forcing term]{Existence of weak solution for mean curvature flow with transport term and forcing term}
\author[K. Takasao]{Keisuke Takasao\\Department of Mathematics/Hakubi Center, Kyoto University, Kitashirakawa-Oiwakecho Sakyo Kyoto 606-8502, Japan}
\date{}
\email{k.takasao@math.kyoto-u.ac.jp}
\begin{document}
\maketitle

\vspace{-0.8cm}
\begin{abstract}
We study the mean curvature flow with given non-smooth transport term and forcing term, in suitable Sobolev spaces. We prove the global existence of the weak solutions for the mean curvature flow with the terms, by using the modified Allen-Cahn equation that holds useful properties such as the monotonicity formula.
\end{abstract}
\section{Introduction}
Let $d\geq 2$ and $\Omega$ be the torus, that is, $\Omega := \mathbb{T}^d =(\mathbb{R}/\mathbb{Z})^d$. Assume that $U_t \subset \Omega$ is an open set with a smooth boundary $M_t:=\partial U_t$ for $t\geq 0$. A family $\{ M_t \}_{t\geq 0}$ of hypersurfaces in $\Omega$ is called a mean curvature flow (MCF) with transport term and forcing term if the normal velocity vector $v$ of $M_t$ satisfies the following:
\begin{equation}
v=h+(u\cdot \nu+g)\nu \quad \text{on} \ M_t, \ t>0,
\label{mcf}
\end{equation}
where $u:\Omega \times (0,\infty) \to \mathbb{R}^d$ and $g:\Omega \times (0,\infty) \to \mathbb{R}$ are given functions, $\cdot$ is the inner product in $\mathbb{R}^d$, $h$ and $\nu$ are the mean curvature vector and the inner unit normal vector of $M_t$, respectively. 
In \cite{liu-sato-tonegawa, liu-sato-tonegawa2}, 
they considered the MCF with transport term ($g\equiv 0$) 
to study the incompressible and viscous non-Newtonian two-phase fluid flow
introduced by Liu and Walkington \cite{liu-walkington}.
The MCF with forcing term ($u\equiv 0$) corresponds to 
the crystal growth (see \cite{bcf,giga2006,soner1993}). 

In the case of $u\equiv 0$ and $g\equiv 0$, Brakke~\cite{brakke} defined the general weak solution (Brakke flow) for \eqref{mcf} via the geometric measure theory and proved the global existence. 
Ilmanen~\cite{Ilmanen} also showed the global existence of the Brakke flow by the phase field method.
Recently, Kim and Tonegawa~\cite{kim-tonegawa} showed the global existence of the multi-phase MCF in the sense of the Brakke flow(see also \cite{tonegawa-book}). 
For other weak solutions, it is well-known that \cite{chen-giga-goto} and \cite{evans-spluck1991} proved the existence of the global unique solution in the sense of viscosity solutions. In addition, about the global existence of the MCF, we also mention \cite{almgren-taylor-wang,ilmanen1994,luckhaus-sturzenhecker}.

 In the case of $u\not\equiv 0$ or $g\not\equiv 0$, Liu, Sato and Tonegawa~\cite{liu-sato-tonegawa} proved the global existence of the weak solution for \eqref{mcf} with $g\equiv 0$ in the sense of the Brakke flow as long as the given transport term $u$ belongs to $L^p_{\text{loc}} ((0,\infty) ;(W^{1,p} (\Omega))^d) $ for $p>(d+2)/2$ and $d=2,3$. Takasao and Tonegawa~\cite{takasao-tonegawa} also proved the existence for more general settings, that is, $d\geq 2$ and $u$ belongs to $L^q _{\text{loc}} ((0,\infty) ;(W^{1,p} (\Omega))^d) $ for $q\in (2,\infty)$ and $p\in (dq/2(q-1),\infty)$ ($p\geq 4/3$ in addition if $d=2$). 
 On the other hand, Mugnai and R{\"o}ger \cite{mugnai-roger2011} showed the global existence of the weak solution called $L^2$-flow for \eqref{mcf} with $u\in L_{\text{loc}} ^2 ((0,\infty) ;(L^\infty (\Omega))^d )$ and $g \in L_{\text{loc}} ^2 ((0,\infty) ; L^\infty (\Omega) )$ for $d=2,3$ (see \cite[Section 5.2]{mugnai-roger2011}).
As explained later in this section, the existence of the weak solution
can be expected for $g$ under the same conditions as \cite{takasao-tonegawa}. 
One motivation in this paper is the generalization of the function space of $g$ 
in the existence theorem for \eqref{mcf}.

Let $\varepsilon \in (0,1)$. In \cite{Ilmanen}, to show the existence of the weak solution for \eqref{mcf} with $u\equiv 0 $ and $g\equiv 0$ in the sense of the Brakke flow, the author studied the following Allen-Cahn equation~\cite{allen-cahn}: 
\begin{equation}
\left\{ 
\begin{array}{ll}
\varepsilon \varphi ^{\varepsilon}_t =\varepsilon \Delta \varphi ^{\varepsilon} -\dfrac{W' (\varphi ^{\varepsilon})}{\varepsilon } ,& (x,t)\in \Omega \times (0,\infty),  \\
\varphi ^{\varepsilon} (x,0) = \varphi _0 ^{\varepsilon} (x) ,  &x\in \Omega,
\end{array} \right.
\label{ac0}
\end{equation}
where $W$ is the double-well potential, such as $W(s) =(1-s^2)^2/2$. 

Set 
$d\mu _t ^\varepsilon := 
\frac{1}{\sigma} \Big( \frac{\varepsilon |\nabla \varphi ^\varepsilon (x,t)|^2 }{2} 
+\frac{W(\varphi ^\varepsilon (x,t))}{\varepsilon}\Big) \, dx$ and
$d\tilde \mu _t ^\varepsilon := 
\frac{\varepsilon }{\sigma} |\nabla \varphi ^\varepsilon (x,t)|^2 \, dx$,
where $\sigma = \int _{-1} ^1 \sqrt{2W (s)} \, ds$. 
These measures correspond to the Hausdorff measure $\mathcal{H}^{d-1} \lfloor_{M_t ^\varepsilon}$,
where $M_t ^\varepsilon =\{ x \in \Omega \, | \, \varphi ^\varepsilon (x,t)=0 \}$.
By integration by parts, we have
\begin{equation*}
\begin{split}
\frac{d}{dt} \int _\Omega \phi \, d\mu _t ^\varepsilon  
= \int_{\Omega}  \nabla \phi \cdot h^\varepsilon -\phi | h^\varepsilon |^2 \, d \tilde{\mu}_t ^\varepsilon 
+\int _{\Omega} \phi _t \, d\mu _t ^\varepsilon \quad \text{for any} \ 
\phi \in C_c ^1 (\Omega \times (0,\infty) ; [0,\infty )),
\end{split}
\end{equation*}
where $h^\varepsilon = \frac{ -\Delta \varphi ^{\varepsilon} 
-W' (\varphi ^{\varepsilon})/\varepsilon ^2 }{|\nabla \varphi ^\varepsilon|} \cdot 
\frac{\nabla \varphi ^\varepsilon}{|\nabla \varphi ^\varepsilon|}$.
The vector-valued function $h^\varepsilon$ is the approximation of the mean curvature vector for $M_t ^\varepsilon$.
Formally we obtain the limit $M_t =\lim_{\varepsilon \to 0} M_t ^\varepsilon$ and the following Brakke's inequality(see \cite{Ilmanen}):
\begin{equation}
\begin{split}
\int _{M_t} \phi \, d\mathcal{H}^{d-1} \Big|_{t=t_1} ^{t_2} 
\leq \int _{t_1} ^{t_2} \int_{M_t}  \nabla \phi \cdot h -\phi | h|^2 +\phi _t \, d\mathcal{H}^{d-1}dt 
\end{split}
\label{brakke-ineq}
\end{equation}
for any $0\leq t_1<t_2 <\infty$ and
$\phi \in C_c ^1 (\Omega \times [0,\infty) ; [0,\infty ))$. Note that 
$\int_{M_t} \phi |h|^2 \, d\mathcal{H}^{d-1} \leq \liminf _{\varepsilon \to 0} \int_{\Omega}\phi |h^\varepsilon| ^ 2\, d\mu_t ^\varepsilon$ implies the inequality of
\eqref{brakke-ineq}.
The Brakke flow is the weak solution characterized by \eqref{brakke-ineq}.
If the solution is smooth, 
then the definition of the Brakke flow and the MCF
are equivalent (see \cite[Proposition 2.1]{tonegawa-book}). In addition, for any initial data $M_0$, 
there exists the trivial solution $\{ M_t \} _{t\geq 0}$ defined by $M_t=\emptyset$ for $t>0$.
Therefore, it is necessary to ensure that the weak solution obtained is non-trivial.
One advantage of the existence theorem via \eqref{ac0} 
is that one can prove the existence of non-trivial solutions, 
since $| \{ x\in \Omega \, | \, \lim _{\varepsilon \to 0} \varphi ^\varepsilon (x,t) =1 \} | $
is a $C^{\frac12}$ function with respect to $t$
(see \cite[Proposition 8.3]{takasao-tonegawa}).

The above discussion requires $\lim_{\varepsilon \to 0} \mu _t ^\varepsilon = \lim_{\varepsilon \to 0} \tilde \mu _t ^\varepsilon$ as Radon measures, 
so the following property is important:
\begin{equation}
\int _{\Omega} \left| \frac{\varepsilon |\nabla \varphi ^\varepsilon (x,t)|^2 }{2} - \frac{W(\varphi ^\varepsilon(x,t))}{\varepsilon}\right| \, dx \to 0 \quad \text{as} \ \varepsilon \downarrow 0
\label{vanish}
\end{equation}
for a.e. $t\geq 0$. The property \eqref{vanish} is called the vanishing of the discrepancy measure (see Definition \ref{defmeas} below) and is also important to show the rectifiability 
of the limit measure $\lim _{\varepsilon \to 0} \mu_t ^\varepsilon$ (see \cite[Section 9.3]{Ilmanen})
and the existence of the $L^2$-flow. 
To prove \eqref{vanish}, Ilmanen~\cite{Ilmanen} showed the non-positivity of the discrepancy measure, that is,
\begin{equation}
 \frac{\varepsilon |\nabla \varphi ^\varepsilon (x,t)|^2 }{2} - \frac{W(\varphi ^\varepsilon(x,t))}{\varepsilon} \leq 0 , \quad (x,t) \in \Omega \times [0,\infty),
\label{negativity}
\end{equation}
for \eqref{ac0} under several suitable assumptions.
Using \eqref{negativity}, one can obtain 
an estimate called monotonicity formula, that is,
\begin{equation}
\frac{d}{dt} \int _{\mathbb{R}^d} \rho _{y,s} (x,t) \, d\mu_t ^\varepsilon (x) 
\leq  \int_{\mathbb{R}^d} \frac{\rho _{y,s} (x,t)}{2(s-t)} 
\left( \frac{\varepsilon |\nabla \varphi ^\varepsilon (x,t)|^2 }{2} 
- \frac{W(\varphi ^\varepsilon(x,t))}{\varepsilon}\right) \, dx \leq 0.
\label{intro-monotonicity1}
\end{equation}
Here
\[ \rho _{y,s} (x,t) := \frac{1}{(4\pi (s-t))^{\frac{d-1}{2}}} e^{-\frac{|x-y|^2}{4(s-t)}}, \qquad t<s, \ x,y\in \mathbb{R}^d \]
and, $\varphi ^\varepsilon$ and $\mu _t ^\varepsilon$ are extended periodically to $\mathbb{R}^d$.
The function $\rho$ is called the backward heat kernel. 
Note that 
$\rho$ converges to the Dirac delta function $\delta _y$ for a $(d-1)$-dimensional surface as $t \to s$. 
Assume that $D:=\sup_{\varepsilon \in (0,1)} \mu _0 ^\varepsilon (\Omega)<\infty$. 
The non-positivity \eqref{negativity} and the monotonicity formula \eqref{intro-monotonicity1} implies that 
there exists $C>0$ depending only on $D$ such that
\begin{equation}
\lim _{\delta \downarrow 0} 
\int_{0} ^{s-\delta} \frac{1}{s-t} \int_\Omega \rho _{y,s} (x,t) 
\left| \frac{\varepsilon |\nabla \varphi ^\varepsilon (x,t)|^2 }{2} 
- \frac{W(\varphi ^\varepsilon(x,t))}{\varepsilon}\right| \, dx dt \leq C
\label{intro-monotonicity2}
\end{equation}
for any $(y,s) \in \mathbb{R}^d \times [0,\infty)$. 
Roughly speaking, if \eqref{vanish} does not hold, then the left hand side of \eqref{intro-monotonicity2} 
is unbounded for some $(y,s)$,
since $\int _0 ^s \frac{1}{s-t} \, dt =\infty$.
Therefore \eqref{negativity} is important property in this discussion. 
In this paper, we use the results of \cite[Proposition 4.9]{roger-schatzle} to obtain \eqref{vanish}
(see Theorem \ref{rs} below and note that the result needs $d=2$ or $3$). 
So we do not use this argument in this paper, 
but \eqref{negativity} is still important in the case of $d\geq 4$, and to 
estimate $\int \rho \, d\mu _t ^\varepsilon$
and the upper bound of the density for the measure $\mu _t ^\varepsilon$ 
(see Theorem \ref{theorem-density} below). 

In \cite{liu-sato-tonegawa,takasao-tonegawa}, to  consider the MCF with additional transport term, they studied the following: 
\begin{equation}
\left\{ 
\begin{array}{ll}
\varepsilon \varphi ^{\varepsilon}_t =\varepsilon \Delta \varphi ^{\varepsilon} -\dfrac{W' (\varphi ^{\varepsilon})}{\varepsilon } - \varepsilon u^{\varepsilon} \cdot \nabla \varphi ^\varepsilon ,& (x,t)\in \Omega \times (0,\infty),  \\
\varphi ^{\varepsilon} (x,0) = \varphi _0 ^{\varepsilon} (x) ,  &x\in \Omega,
\end{array} \right.
\label{ac1}
\end{equation}
where $u^\varepsilon$ is the smooth approximation of $u$. In \cite{mugnai-roger2011}, they considered the following Allen-Cahn equation with forcing term:
\begin{equation}
\left\{ 
\begin{array}{ll}
\varepsilon \varphi ^{\varepsilon}_t =\varepsilon \Delta \varphi ^{\varepsilon} -\dfrac{W' (\varphi ^{\varepsilon})}{\varepsilon } - G^\varepsilon ,& (x,t)\in \Omega \times (0,\infty),  \\
\varphi ^{\varepsilon} (x,0) = \varphi _0 ^{\varepsilon} (x) ,  &x\in \Omega,
\end{array} \right.
\label{ac4}
\end{equation}
where $G^\varepsilon$ is smooth and satisfies $\sup _{\varepsilon >0}\int _0 ^T \int _\Omega \varepsilon^{-1} |G^\varepsilon|^2 \, dx dt<\infty$. Let $g^\varepsilon$ be the smooth approximation of $g$. Note that substituting $\varepsilon u^\varepsilon \cdot \nabla \varphi ^\varepsilon +  g^\varepsilon \sqrt{2W(\varphi^\varepsilon)}$ into $G^\varepsilon$, we obtain \eqref{mcf} as $\varepsilon \to 0$ in the sense of $L^2$-flow (see \cite[Section 5.2]{mugnai-roger2011}).

In the case of $u^\varepsilon \not \equiv 0$ or $g^\varepsilon \not \equiv 0$, the property \eqref{negativity} does not hold for \eqref{ac1} and \eqref{ac4}, generally. 
Therefore, the proof of \eqref{vanish} in \cite{Ilmanen} is not applicable 
to \eqref{ac1} or \eqref{ac4}.
To prove \eqref{vanish}, \cite{mugnai-roger2011} used the result of \cite[Proposition 4.9]{roger-schatzle} (see Theorem \ref{rs} below). On the other hand, in \cite{liu-sato-tonegawa,takasao-tonegawa}, they used weaker estimates than \eqref{negativity} to obtain \eqref{intro-monotonicity2} and \eqref{vanish}. However, we can not apply the technique for the case of $g^\varepsilon \not \equiv 0$ directly (see Remark \ref{rem4.2} below). 
Another motivation for this paper is to propose the new phase field method 
that has the property \eqref{negativity} 
even when there are transport term and forcing term.

Let $q^\varepsilon=q^\varepsilon (r)$ be a solution for
\begin{equation}
\frac{\varepsilon (q^{\varepsilon} _{r})^2 }{2}= \frac{W (q ^{\varepsilon})}{\varepsilon },\quad r\in \mathbb{R}, \quad q^\varepsilon (\pm \infty)=\pm 1, \quad q^\varepsilon (0)=0, \ \text{and} \ \quad q^\varepsilon_r (r)>0, \quad r\in \mathbb{R}.
\label{q}
\end{equation}
For example, if $W(s)=(1-s^2)^2/2$, then $q^\varepsilon (r) =\tanh (r/\varepsilon)$ satisfies \eqref{q}.
Set $T>0$. In this paper, we consider the following modified Allen-Cahn equation with transport term and forcing term:
\begin{equation}
\left\{ 
\begin{array}{ll}
\varepsilon \varphi ^{\varepsilon}_t =\varepsilon \Delta \varphi ^{\varepsilon} -\dfrac{W' (\varphi ^{\varepsilon})}{\varepsilon }  -\varepsilon u^{\varepsilon} \cdot \nabla \varphi ^\varepsilon -(g^\varepsilon + L^\varepsilon r^\varepsilon )\sqrt{2W (\varphi ^{\varepsilon})}  ,& (x,t)\in \Omega \times (0,T),  \\
\varphi ^{\varepsilon} (x,0) = \varphi _0 ^{\varepsilon} (x) ,  &x\in \Omega,
\end{array} \right.
\label{ac2}
\end{equation}
where 
\[
L^\varepsilon :=\left( 2 \sup _{(x,t) \in \Omega \times (0,T)} |\nabla u^\varepsilon (x,t)| + \sup _{(x,t) \in \Omega \times (0,T)} |\nabla g^\varepsilon (x,t)|\right)
\] 
and $r^\varepsilon=r^\varepsilon (x,t)$ is given by $\varphi^\varepsilon(x,t) = q^\varepsilon (r ^\varepsilon (x,t))$. Note that if there exists $(x,t)\in \Omega \times (0,T)$ such that $|\varphi ^\varepsilon (x,t)| =1$, then $r^\varepsilon$ is not well-defined. However, that case does not occur under suitable conditions (see Proposition \ref{prop-existence} below). Define 
\[ 
f^\varepsilon := -(u^\varepsilon \cdot \nabla r^\varepsilon) -g^\varepsilon -L^\varepsilon r^\varepsilon .
\]
We remark that by \eqref{q}, the first equation of \eqref{ac2} is equal to
\begin{equation}
\varepsilon \varphi ^{\varepsilon}_t =\varepsilon \Delta \varphi ^{\varepsilon} -\dfrac{W' (\varphi ^{\varepsilon})}{\varepsilon } + f^\varepsilon \sqrt{2W (\varphi ^{\varepsilon})} .
\label{ac3}
\end{equation}
By adding the forcing term $- L^\varepsilon r^\varepsilon \sqrt{2W (\varphi ^{\varepsilon})}$, 
we can obtain \eqref{negativity}, 
because if  the term is added to the phase field method, then an argument similar to that in \cite{Ilmanen} 
(the maximum principle for $w^\varepsilon :=|\nabla r^\varepsilon|^2-1$)
can be used (see Lemma \ref{lem-negative} below).
In addition, the additional term is very small in the framework of the phase field method under several assumptions (see Remark \ref{rem4.3} below). Roughly speaking, the reason is that $r^\varepsilon\approx 0$ near the zero level set of $\varphi ^\varepsilon$.
Therefore we can obtain the monotonicity formula and the convergence of the solutions for \eqref{ac2} to 
the global weak solution for \eqref{mcf}, with $d= 2, 3$, and $u \in L^q _{\text{loc}} ((0,\infty) ;(W^{1,p} (\Omega))^d) $ and $g \in L^q _{\text{loc}} ((0,\infty) ;W^{1,p} (\Omega)) $, where $q\in (2,\infty)$ and $p\in (dq/2(q-1),\infty)$ ($p\geq 4/3$ in addition if $d=2$). The precise statements of the main results are described in Section 3.
The condition $p\in (dq/2(q-1),\infty)$ is natural in the following sense
(same argument is mentioned in \cite{takasao-tonegawa}).
Let $\lambda >0$ and consider the standard parabolic rescaling, that is, $\tilde x =\frac{x}{\lambda}$
and $\tilde t= \frac{t}{\lambda^2}$. The functions $u$ and $g$ correspond to the velocity of $M_t$,
therefore rescaled functions should be $\tilde u(\tilde x,\tilde t) = \lambda u (x,t)$ and
$\tilde g(\tilde x,\tilde t) = \lambda g (x,t)$, since $\frac{\tilde x}{ \tilde t}= \lambda \frac{x}{t}$. 
We compute
\[
\Big( \int _0 ^\infty \Big( \int _{\mathbb{R} ^d} |\nabla w |^p dx \Big)^{\frac{q}{p}} dt \Big)^{\frac{1}{q}}
= \lambda ^{\frac{d}{p} +\frac{2}{q} -2}
\Big( \int _0 ^\infty \Big( \int _{\mathbb{R} ^d} |\nabla_{\tilde x} \tilde{w}|^p d\tilde{x}\Big)^{\frac{q}{p}} d\tilde{t} \Big)^{\frac{1}{q}},
\]
where $w=u$ or $g$.
The condition $p\in (dq/2(q-1),\infty)$ is equivalent to $\frac{d}{p} +\frac{2}{q} -2<0$.
Hence the transport term and forcing term can be regarded as perturbations.

About the phase field method for the MCF, there are a huge number of results and we mention \cite{bronsard-kohn,xchen1992, evans-soner-souganidis, giga2006, rubinstein-sternberg-keller,soner1995,soner1997} and references therein.

The paper is organized as follows. In Section 2, we set our notations and definitions. In Section 3, we explain the main results of this paper.
In Section 4, first we show the non-positivity of the discrepancy measure and the monotonicity formula. Then we prove the upper bound of the density of $\mu _t ^\varepsilon$ (Theorem \ref{theorem-density}) and the existence theorem for \eqref{mcf} (Theorem \ref{theorem-existence}). In Section 5, we explain the several theorems used in this paper as a supplement.

\section{Notation and definitions}
Throughout this paper, we consider the case of $\Omega = \mathbb{T}^d =(\mathbb{R}/\mathbb{Z})^d$. For $r>0$ and $y\in \mathbb{R}^k$ we define $B_r ^k (y):=\{ x\in \mathbb{R}^k \, | \, |x-y|<r \}$. Set $\omega _k :=\mathscr{L}^k (B_1 ^k (0))$. We denote
\[ D(t) := \max \left\{ 1,\mu _t ^\varepsilon (\Omega), \sup _{B_r ^d (x) \subset \Omega}\frac{\mu _t ^\varepsilon (B_r ^d (x))}{\omega_{d-1} r^{d-1}} \right\}, \quad t \in [0,\infty).\]
\begin{definition}\label{defmeas}
Set $\sigma := \int_{-1} ^1 \sqrt{2W(s)} \, ds $. Let $\varphi ^\varepsilon$ be a solution for \eqref{ac2}. We define a Radon measure $\mu _t ^{\varepsilon}$ and $\xi _t ^{\varepsilon}$ by
\begin{equation*}
\mu _t ^{\varepsilon}(\phi) := \frac{1}{\sigma}\int _{\Omega} \phi(x) \Big( \frac{\varepsilon |\nabla \varphi ^{\varepsilon} (x,t)|^2}{2} + \frac{W (\varphi ^{\varepsilon} (x,t))}{\varepsilon} \Big) dx
\end{equation*}
and
\[  
\xi _t ^{\varepsilon}(\phi) := \frac{1}{\sigma}\int _{\Omega} \phi (x) \Big( \frac{\varepsilon |\nabla \varphi ^{\varepsilon} (x,t)|^2}{2} - \frac{W (\varphi ^{\varepsilon} (x,t))}{\varepsilon} \Big) dx
\]
for any $\phi \in C_c (\Omega)$. The measure $\xi_t ^\varepsilon$ is called the discrepancy measure.
\end{definition}

In this paper, we suppose that a function $W$ satisfies the following:
\begin{equation}
W:\mathbb{R} \to [0,\infty) \ \text{is smooth} \quad \text{and} \quad W(\pm 1) = W' (\pm 1) =0.
\label{w1}
\end{equation}
\begin{equation}
\text{For some} \ \alpha_1 \in (-1,1), \ W'<0 \ \ \text{on} \ \ (\alpha_1,1) \quad \text{and} \quad W'>0 \ \ \text{on} \ \ (-1,\alpha_1).
\label{w2}
\end{equation}
\begin{equation}
\text{There exist} \ \alpha_2 \in (0,1) \ \text{and} \ \ \kappa>0, \ \text{such that} \quad W''(s)>0 \ \ \text{for any} \ \ \alpha_2 \leq |s| \leq 1.
\label{w3}
\end{equation}
\begin{equation}
\text{There exists} \ \Cl{const:w}>0 \ \ \text{such that} \quad (q^{-1}(s))^2 W(s) \leq \Cr{const:w} \ \ \text{for any} \ \ |s| <1.
\label{w4}
\end{equation}
Here $q$ is a solution for \eqref{q} with $\varepsilon =1$ and $q^{-1}$ is the inverse function of $q$. For example, $W(s)= (1-s^2)^2/2$ satisfies \eqref{w1}, \eqref{w2}, \eqref{w3},  and \eqref{w4}. We remark that $q(r) = \tanh r$ in the case of $W(s)= (1-s^2)^2/2$. 

Next we recall several definitions and notations from the geometric measure theory and refer to \cite{allard,brakke,federer,giusti,simon,tonegawa-book} for more details. 
For a set $U \subset \Omega$ with finite perimeter, we denote the reduced boundary by $\partial ^\ast U$, and the total variation measure of the distributional derivative $\chi _U$ is denoted by $\|\nabla \chi _ U\|$.
Let $\mu $ be a Radon measure on $\Omega$. We call $\mu $ $k$-rectifiable if $\mu $ is represented by $\mu = \theta \mathcal{H} ^k \lfloor M$, that is, $\int _{\Omega} \eta \, d\mu = \int _M \eta \theta \, d\mathcal{H}^k$ for any $\eta \in C_c (\Omega)$ (see \cite[Section 3.5]{allard} or \cite[Section 15]{simon}), where $M\subset \Omega$ is a $\mathcal{H}^k$-measurable countably $k$-rectifiable set, and $\theta \in L^1 _{loc} (\mathcal{H}^k \lfloor M)$ is a positive valued function $\mathcal{H}^k$-a.e. on $M$. In addition, if $\theta$ is positive and integer-valued $\mathcal{H}^k$-a.e. on $M$ then we call $\mu$ $k$-integral. Especially, if $\theta\equiv 1$, we say $\mu $ has unit density. Let $T$ be a hyper plane in $\mathbb{R}^d$ with $0 \in T$ and $\nu$ be the unit normal vector of $T$. We also use $T$ to denote the orthogonal projection $\mathbb{R}^d\to T$, that is, $T= \text{Id} -\nu \otimes \nu$, where $\text{Id}$ is the identity matrix.

Assume that $M$ is a countably $(d-1)$-rectifiable and $\mathcal{H}^{d-1}$-measurable subset of $\Omega$ and $\theta\in L^1_{loc} (\mathcal{H}^{d-1} (M))$ is a positive function. For a Radon measure $\mu:=\theta \mathcal{H}^{d-1} \lfloor _{M}$, $h$ is called a generalized mean curvature vector if
\[ \int  _{\Omega} \text{div} _{M} \, \Phi \, d\mu =- \int_{\Omega}  h \cdot \Phi \, d\mu \]
holds for any $\Phi \in C_c ^1 (\Omega ; \mathbb{R} ^d)$ (see \cite[Section 2.9]{brakke} or \cite[Section 16]{simon}).

\bigskip

The following definition is similar to the formulation of the Brakke flow~\cite{brakke}:
\begin{definition}[$L^2$-flow~\cite{mugnai-roger2008}]
Let $T>0$ and $\{ \mu _t \}_{t\in (0,T)}$ be a family of Radon measures on $\Omega$. Set $d\mu :=d\mu_t dt$. We call $\{ \mu _t \}_{t\in (0,T)}$ an $L^2$-flow if the following holds:
\begin{enumerate}
\item $\mu _t$ is $(d-1)$-integral and has a generalized mean curvature vector $h \in L^2 (\mu _t;\mathbb{R}^d)$ a.e. $t\in (0,T)$, 
\item and there exist $C>0$ and a vector $v\in L^2 (0,T;(L^2 (\mu _t) )^d)$ such that
\begin{equation}
v(x,t)\perp T_x \mu _t \quad \text{for} \ \mu \text{-a.e.} \ (x,t) \in \Omega \times (0,T)
\label{velo1}
\end{equation}
and
\begin{equation}
\Big| \int _0 ^T \int _{\Omega} (\eta _t + \nabla \eta \cdot v ) \, d\mu_t dt \Big| \leq C \| \eta \|_{\infty}
\label{velo2}
\end{equation}
for any $\eta \in C_c ^1 (\Omega\times (0,T))$. Here $T_x \mu _t $ is the approximate tangent plane of $\mu _t$ at $x$.
\end{enumerate}
In addition, the above vector $v\in L^2 (0,T;(L^2 (\mu _t) )^d)$ is called a generalized velocity vector.
\end{definition}
\begin{remark}
If $\{ \mu _t \}_{t\in (0,T)}$ is an integral Brakke flow, then it is also $L^2$-flow (see\cite[Section 2.5]{bertini-butta-pisante}).
\end{remark}
\section{Main results}
In this paper, first we show the non-positivity of the discrepancy measure and the upper bound of the density for the measure $\mu _t ^\varepsilon$.
\begin{theorem}\label{theorem-density}
Assume that $T>0$, $d\geq 2$, $2<q<\infty$, $p\in[\frac{2d}{d+1} ,\infty) \cap ( \frac{dq}{2(q-1)}, \infty)$, and 
\begin{equation}
0< \gamma<\frac12.
\label{assumption-gamma}
\end{equation}
Suppose that $\varphi ^\varepsilon$ is a classical solution for \eqref{ac2} with $\max_{x\in \Omega}|\varphi _0 ^\varepsilon (x)| <1$ and
\begin{equation*}
 \frac{\varepsilon |\nabla \varphi ^\varepsilon _0 (x)|^2 }{2} - \frac{W(\varphi ^\varepsilon _0(x) )}{\varepsilon} \leq 0 , \quad x \in \Omega,
\end{equation*}
and $u^\varepsilon \in ( C ^\infty (\Omega \times [0,T]) )^d$, $g^\varepsilon \in  C ^\infty (\Omega \times [0,T]) $ with
\begin{equation}
L^\varepsilon = 2 \sup _{(x,t) \in \Omega \times (0,T)} |\nabla u^\varepsilon (x,t)| + \sup _{(x,t) \in \Omega \times (0,T)} |\nabla g^\varepsilon (x,t)| \leq \varepsilon ^{ -\gamma},
\label{eq3.2}
\end{equation}
\begin{equation}
\| u^\varepsilon \|^2 _{L^q ([0,T] ; (W^{1,p} (\Omega))^d)} + \| g^\varepsilon \|^2 _{L^q ([0,T] ; W^{1,p} (\Omega))}<\infty,
\label{eq3.3}
\end{equation}
and there exists $D_0 >0$ such that
\begin{equation}
D(0)\leq D_0.
\end{equation}
Then the following hold:
\begin{enumerate}
\item The non-positivity \eqref{negativity} holds for any $(x,t) \in \Omega \times [0,T)$.
\item There exist $D_1 >0$ and $\epsilon \in (0,1) $ such that
\begin{equation}
\sup _{0\leq	 t \leq T} D(t)\leq D_1, \quad \varepsilon \in (0,\epsilon).
\label{density-ineq}
\end{equation}
\end{enumerate}
\end{theorem}

\bigskip

\begin{remark}
Similar result about the density bound 
has been obtained in \cite{liu-sato-tonegawa,takasao-tonegawa}. 
The difficult part of the proof of the density bound is
the estimate of the positive part of the discrepancy measure.
Therefore, one of the advantages of this paper is that 
the phase field method for \eqref{mcf} with the non-positivity \eqref{negativity} was obtained.
The property is also useful for obtaining the monotonicity formula and the vanishing of 
the discrepancy measure (see Lemma \ref{propmono} below).
In addition, in the case of $g^\varepsilon \not=0$,
it will be difficult to obtain the estimate of the discrepancy measure 
via the phase field method without the additional term  
$- L^\varepsilon r^\varepsilon \sqrt{2W (\varphi ^{\varepsilon})}$ 
(see Remark \ref{rem4.2} below).
\end{remark}
\begin{remark}
For the regularity corresponding to \eqref{eq3.2}, 
\[
\sup _{\Omega \times [0,T]} |  u^\varepsilon | \leq \varepsilon ^{ -\gamma} \quad \text{and} \quad \sup _{\Omega \times [0,T]} | \nabla u^\varepsilon | \leq \varepsilon ^{ -(\gamma+1)}
\]
are assumed in \cite{takasao-tonegawa}, where $\gamma \in (0,\frac12)$.
In Theorem \ref{theorem-density}, the estimate of $\sup _{\Omega \times [0,T]} |  u^\varepsilon | $
is not required. However, the assumption for 
$\sup _{\Omega \times [0,T]} | \nabla u^\varepsilon | $
is stronger than that in \cite{takasao-tonegawa}.
\end{remark}

\begin{remark}
The assumption \eqref{eq3.2} is used to prove that the additional term 
$- L^\varepsilon r^\varepsilon \sqrt{2W (\varphi ^{\varepsilon})}$ converges to $0$ (see Remark \ref{rem4.3}), 
and \eqref{eq3.3} is mainly necessary for the $L^2$-estimates of transport term and forcing term (see Lemma \ref{lemma4.6} and Lemma \ref{lem4.7}).
\end{remark}

\bigskip

Set 
\[
  v^\varepsilon = \begin{cases}
    \frac{-\varphi^\varepsilon _t}{|\nabla\varphi^\varepsilon|} \frac{\nabla\varphi^\varepsilon}{|\nabla\varphi^\varepsilon|} & \text{if} \ |\nabla\varphi^\varepsilon|\not=0 ,\\
   \qquad 0 & \text{otherwise}.
  \end{cases}
\]
Let $\Psi _\delta \in C_c ^\infty (B_\delta (0))$ be the Dirac sequence, and $\{ \delta _i \} _{i=1} ^\infty$ 
and $\{ T_i \}_{i=1} ^\infty$ be positive sequences with $\delta _i \to 0$ and 
$T _i \to \infty$ as $i\to \infty$, respectively. 
For $\gamma \in (0,\frac12)$, $u \in L^q _{loc} ([0,\infty) ; (W^{1,p} (\Omega))^d)$, and $g \in L^q _{loc}  ([0,\infty) ; W^{1,p} (\Omega))$,
we choose a positive sequence $\{ \varepsilon _i \} _{i=1} ^\infty $ such that 
$\varepsilon _i \to 0$, 
\begin{equation}
\sup _{\Omega \times [0,T_i]} | \nabla u^{\varepsilon_i} | \leq \varepsilon _i ^{ -\gamma}, 
\quad \text {and} 
\quad \sup _{\Omega \times [0,T_i]} | \nabla g^{\varepsilon_i} | \leq \varepsilon_i ^{ -\gamma} 
\qquad \text{for any} \ i\geq 1,
\label{delta-conv}
\end{equation}
where $u^{\varepsilon_i} :=\Psi _{\delta _i } \ast u$, 
and $g^{\varepsilon _i} :=\Psi _{\delta _i } \ast g$. Note that
\[
u^{\varepsilon_i}  \to u \quad \text{in} \ L^q _{loc} ([0,\infty) ; (W^{1,p} (\Omega))^d)
 \quad \text{and} \quad
g^{\varepsilon_i}  \to g \quad \text{in} \ L^q _{loc}  ([0,\infty) ; W^{1,p} (\Omega)).
\]
 For the solution $\varphi ^{\varepsilon_i}$ for \eqref{ac2} with $\varepsilon =\varepsilon_i$ and $T=T_i$, we define $\varphi ^{\varepsilon _i} (x,t) = 1$ if $t \geq T_i$, for the following theorem.
By using Theorem \ref{theorem-density}, we show the vanishing of the discrepancy measure and the existence of the weak solution for \eqref{mcf}:

\begin{theorem}\label{theorem-existence}
Let $d=2,3$ and $u \in L^q _{loc} ([0,\infty) ; (W^{1,p} (\Omega))^d)$ and $g \in L^q _{loc}  ([0,\infty) ; W^{1,p} (\Omega))$. 
Let $\{ \delta _i \} _{i=1} ^\infty$, $\{ \varepsilon _i \} _{i=1} ^\infty$ and $\{ T_i \} _{i=1} ^\infty$ 
be positive sequences such that \eqref{delta-conv} holds.
Assume that for any $i\geq 1$ all assumptions of Theorem \ref{theorem-density} hold with 
$\varepsilon =\varepsilon _i$, $T=T_i$.
Then there exists a subsequence (we denote $\varepsilon_{i_j}$ by $\varepsilon$ for simplicity) and the following hold:
\begin{enumerate}
\item There exists a family of $(d-1)$-integral Radon measures $\{ \mu _t \}_{t\in [0,\infty)}$ on $\Omega$ such that
\begin{enumerate}
\item[(1a)] $\mu ^\varepsilon \to \mu$ as Radon measures on $\Omega \times [0,\infty)$, where $d\mu =d\mu_t dt$.
\item[(1b)] $\mu ^\varepsilon _t \to \mu _t$ as Radon measures on $\Omega$ for all $t \in [0,\infty)$.
\end{enumerate}
\item There exists $\psi \in BV _{loc}( \Omega \times [0,\infty)) \cap C^{\frac{1}{2}} _{loc} ([0,\infty) ;L^1 (\Omega))$ such that
\begin{enumerate}
\item[(2a)] $\varphi ^{\varepsilon} \to 2\psi -1 \ \ \text{in} \ L^1 _{loc} ( \Omega \times [0,\infty))$ and a.e. pointwise.
\item[(2b)] $\psi =0$ or $1$ a.e. on $\Omega \times [0,\infty)$.
\item[(2c)] $\| \nabla \psi (\cdot,t) \| (\phi) \leq \mu _t (\phi) $ for any $t\in [0,\infty)$ and $\phi \in C_c (\Omega;[0,\infty))$. Moreover $\spt\| \nabla \psi (\cdot,t) \| \subset \spt \mu _t  $ for any $t\in [0,\infty)$.
\end{enumerate}
\item $\xi ^\varepsilon _t \to 0$ as Radon measures on $\Omega$ for a.e. $t \in [0,\infty)$.
\item For any $\Phi \in C_c (\Omega \times[0,\infty);\mathbb{R}^d)$ we have
\begin{equation*}
\begin{split}
\lim _{\varepsilon \to 0} \frac{1}{\sigma}\int _{\Omega \times (0,\infty)} u^\varepsilon \cdot \Phi  \, \varepsilon |\nabla \varphi^\varepsilon|^2dxdt = \int _{\Omega \times (0,\infty)} u \cdot \Phi  \, d\mu.
\end{split}
\label{weakconv1}
\end{equation*}
\item There exists a vector valued function $\tilde g \in L^2 _{loc} (0,\infty ; (L^2 (\mu _t))^d)$ such that
\begin{equation*}
\begin{split}
\lim _{\varepsilon \to 0} \frac{1}{\sigma}\int _{\Omega \times (0,\infty)} g^\varepsilon \sqrt{2W(\varphi^\varepsilon)} \nabla \varphi^\varepsilon \cdot \Phi  \, dxdt = \int _{\Omega \times (0,\infty)} \tilde g \cdot \Phi  \, d\mu
\end{split}
\label{weakconv2}
\end{equation*}
for any $\Phi \in C_c (\Omega \times[0,\infty);\mathbb{R}^d)$.
\item $\{ \mu _t \}_{t\in (0,\infty)}$ is an $L^2$-flow with a generalized velocity vector 
\begin{equation}
v(x,t)=h(x,t)+(\text{Id} -T_{x} \mu _t ) u(x,t) + \tilde g(x,t),
\label{eq3.6}
\end{equation}
where $h$ is the generalized mean curvature vector of $\mu_t$, $T_{x} \mu _t$ is the approximate tangent plane of $\mu _t$ at $x$, and
\begin{equation}
\lim _{\varepsilon \to 0} \int _{\Omega \times (0,\infty)}  v^\varepsilon  \cdot \Phi \, d\mu ^\varepsilon = \int _{\Omega \times (0,\infty)} v \cdot \Phi \, d\mu
\label{conv-v}
\end{equation}
for any $\Phi \in C_c (\Omega \times[0,\infty);\mathbb{R}^d)$. 
Moreover $\spt \tilde g \subset \partial ^{\ast} \{ \psi =1 \}$ and there exists a measurable function $\theta : \partial ^{\ast} \{ \psi =1 \}\to \mathbb{N}$ such that 
\begin{equation}
\tilde g =\frac{1}{\theta} g \nu \qquad \mathcal{H}^{d}\text{-a.e. on} \ \partial ^{\ast} \{ (x,t) \, | \, \psi (x,t) =1 \},
\label{eq:3.8}
\end{equation}
where $\nu (\cdot, t)$ is the inner unit normal vector of $\{ \psi(\cdot,t) =1 \}$ on $\partial ^{\ast} \{ \psi(\cdot,t) =1 \}$.
\end{enumerate}
\end{theorem}

\begin{remark}
The assumption for $d$ comes from Theorem \ref{rs}.
In the case of $d\geq 4$, then we may need several arguments similar to that in \cite{Ilmanen,takasao-tonegawa}.
The term $(\text{Id} -T_{x} \mu _t ) u$ corresponds to $(u\cdot \nu) \nu$ 
if $\mu _t$ is given by a smooth hypersurface.
\end{remark}

\begin{remark}\label{rem3.7}
In \cite[Section 5.2]{mugnai-roger2011}, they showed the existence theorem with \\
$u\in L_{\text{loc}} ^2 ((0,\infty) ;(L^\infty (\Omega))^d )$ and
 $g \in L_{\text{loc}} ^2 ((0,\infty) ; L^\infty (\Omega) )$ for $d=2,3$.
As mentioned in Section 1, natural function spaces are considered in Theorem \ref{theorem-density} and 
Theorem \ref{theorem-existence}.

In the case of $g\equiv 0$, the existence of the weak solution for \eqref{mcf} in the sense of Brakke flow with $u \in L^q _{loc} ([0,\infty) ; (W^{1,p} (\Omega))^d)$ and $d\geq 2$ has already been proven in \cite{takasao-tonegawa}. 
Here, a family of $(d-1)$-integral Radon measures $\{ \mu _t \} _{t\in [0,\infty)}$ is called a Brakke flow with transport term $u$ if
\[
\int _{\Omega} \phi \, d\mu _t \Big| _{t=t_1} ^{t_2}
\leq 
\int_{t_1} ^{t_2} \int _{\Omega} (\nabla \phi -\phi h) \cdot (h +(u\cdot \nu) \nu)+\phi _t \, d\mu _t dt 
\]
holds for any $\phi \in C_c ^1 (\Omega\times [0,\infty) ; [0,\infty))$. 
Note that the regularity of the Brakke flow is also known (see \cite{kasai-tonegawa,tonegawa2014}). 
The main differences of the phase field methods between \cite{takasao-tonegawa} and this paper are 
having or not having the proofs of the estimates of the positive part of the discrepancy measure, and the additional forcing term $- L^\varepsilon r^\varepsilon \sqrt{2W (\varphi ^{\varepsilon})}$. Because the term is very small in the sense of the Brakke flow (see Remark \ref{rem4.3}), 
it is expected that same existence theorem of the Brakke flow in \cite{takasao-tonegawa} ($d\geq 2$) will be obtained 
via the phase field model \eqref{ac2}. 
In addition, \eqref{negativity} would make it easier to prove the vanishing of the discrepancy measure
than that in \cite{takasao-tonegawa}.

However, in the case of $g \not \equiv 0$, it is difficult to consider the weak solution for \eqref{mcf} in the sense of the Brakke flow, since weak convergences of $\nu^\varepsilon$ and $h^\varepsilon$ are insufficient to make sense of the convergence
\[
\int\phi  g^\varepsilon \nu ^\varepsilon \cdot h^\varepsilon \, d\tilde \mu _t ^\varepsilon \to \int \phi g \nu  \cdot h \, d\mu _t \quad \text{for any} \ \phi \in C_c (\Omega \times [0,\infty)),
\]
where $\nu ^\varepsilon =\frac{\nabla \varphi ^\varepsilon}{|\nabla \varphi ^\varepsilon|}$, $h^\varepsilon = \frac{-\Delta \varphi ^\varepsilon +\frac{W'( \varphi ^\varepsilon)}{\varepsilon} }{|\nabla \varphi ^\varepsilon|} \nu ^\varepsilon$, and $d\tilde \mu _t ^\varepsilon=\frac{\varepsilon}{\sigma} |\nabla \varphi ^\varepsilon|^2 dx$. In particular, when $\mu _t$ is not a unit density measure, the treatment of the orientation of $\nu$ is a problem.
On the other hand, this problem does not occur when $L^2$-flow is considered, because 
the computation of the inner product is not necessary in the definition of the $L^2$-flow and
the characterization of the generalized velocity \eqref{eq3.6}.
\end{remark}

\begin{remark}
Regarding energy estimates, there is no difference in the handling of transport term and forcing term. 
However, regarding convergence, the forcing term converges with respect to 
the measure $\| \nabla \psi (\cdot,t)\|$ (see \eqref{eq4.34}). The function $\theta$ in \eqref{eq:3.8} is the inverse
of the Radon-Nikodym Derivative $\frac{d\| \nabla \psi (\cdot,t) \|}{d\mu_t}$.

\end{remark}

\section{Proof of main theorems}
In this section, we assume all the assumptions of Theorem \ref{theorem-density}. First we prove the well-posedness of the phase field model \eqref{ac2}. Next we show the monotonicity formula via the arguments in \cite{Ilmanen} and the upper bound of the density of $\mu_t ^\varepsilon$ by using the arguments in \cite{liu-sato-tonegawa,takasao-tonegawa}. The upper bound estimates, Theorem \ref{mr}, and standard measure theoretic arguments imply the existence theorem.

\subsection{Well-posednes of \eqref{ac2} }\mbox{}\\[0.2cm]
Let $\delta \in (0,1)$ and $r^\varepsilon _\delta : \mathbb{R} \to \mathbb{R}$ be a $C^\infty$ function such that 
\[
  r^\varepsilon _\delta ( s ) = 
  \begin{cases}
    (q^\varepsilon)^{-1} (-1+\delta) -1 & \text{if} \ s<-1,\\
    (q^\varepsilon)^{-1} (s)  & \text{if} \ s\in [-1+\delta , 1-\delta] ,\\
     (q^\varepsilon)^{-1} (1-\delta) +1 & \text{if} \ s>1.
  \end{cases}
\]

From the definition of $r^\varepsilon$ in \eqref{ac2}, we need the a priori estimate $\varphi ^\varepsilon (x,t) \in (-1,1)$
for any $(x,t) \in \Omega \times [0,T)$. Therefore first we consider the following modified equation:
\begin{equation}
\left\{ 
\begin{array}{ll}
\varepsilon \varphi ^{\varepsilon}_t =\varepsilon \Delta \varphi ^{\varepsilon} -\dfrac{W' (\varphi ^{\varepsilon})}{\varepsilon }  -\varepsilon u^{\varepsilon} \cdot \nabla \varphi ^\varepsilon -(g^\varepsilon + L^\varepsilon r^\varepsilon _\delta (\varphi ^\varepsilon) )\sqrt{2W (\varphi ^{\varepsilon})}  ,& (x,t)\in \Omega \times (0,T),  \\
\varphi ^{\varepsilon} (x,0) = \varphi _0 ^{\varepsilon} (x) ,  &x\in \Omega.
\end{array} \right.
\label{ac5}
\end{equation}

The estimate $\varphi ^\varepsilon  \in (-1,1)$ can be obtained as follows from the maximum principle.

\begin{lemma}\label{prop4.1}
Let $T>0$ and $a \in (0,1)$. Then there exists $\delta \in (0,1)$ such that the following hold: Let $\varphi ^\varepsilon$ be a classical solution for \eqref{ac5} with $\delta>0$ and
$\max _{x\in \Omega} |\varphi ^\varepsilon _0 (x) | \leq 1-a$. Then
$\sup _{(x,t) \in \Omega \times [0,T)} |\varphi ^\varepsilon (x,t)| \leq 1-\delta$.
Moreover, $\varphi ^\varepsilon$ is also a solution for \eqref{ac2} in $\Omega \times [0,T)$.
\end{lemma}

\begin{proof}
Let $\varphi ^\varepsilon$ be a classical solution for \eqref{ac5} with $\delta>0$ and
$\max _{x\in \Omega} |\varphi ^\varepsilon _0 (x) | \leq 1-a$.
By the definition, $r^\varepsilon _\delta (\varphi ^\varepsilon (x,t)) =r^\varepsilon (x,t) $ 
if $|\varphi ^\varepsilon (x,t)| \leq 1-\delta$. So we only need to prove 
$\sup _{(x,t) \in \Omega \times [0,T)} |\varphi ^\varepsilon (x,t)| \leq 1-\delta$.

By the maximum principle, we obtain $\sup _{x\in \Omega, t\in[0,T)} |\varphi ^\varepsilon (x,t)| \leq 1$
easily. Assume that there exists $(x,t) \in \Omega \times [0,T)$ such that $ \varphi ^\varepsilon (x,t) =1$.
Then $T_1 := \inf \{ t \in (0,T] \, | \, \varphi (x,t)=1 \ \text{for some} \ x \in \Omega  \}<T$.
Note that $r^\varepsilon (x,t) = (q^\varepsilon)^{-1} (\varphi ^\varepsilon (x,t))$ is well-defined for any $(x,t) \in \Omega \times [0,T_1)$.

Set $h (q):=\sqrt{2W (q)}$ for $q \in \mathbb{R}$. By \eqref{q} we obtain
\begin{equation}
q^{\varepsilon} _r =\frac{ h (q^{\varepsilon})}{\varepsilon} \qquad \text{and} \qquad q^{\varepsilon} _{rr}= \frac{( h (q^{\varepsilon}) )_r}{\varepsilon}=\frac{h _q ( q^{\varepsilon} ) }{\varepsilon} q^{\varepsilon} _r.
\label{q2'} 
\end{equation}
By \eqref{q}, \eqref{ac5}, and \eqref{q2'} we have
\begin{equation*}
\begin{split}
 q_r ^{\varepsilon} r_t ^\varepsilon  &= q_r ^{\varepsilon} \Delta r^\varepsilon  + q_{rr} ^{\varepsilon} |\nabla r^\varepsilon  |^2 -q_{rr} ^{\varepsilon} -(u^\varepsilon \cdot \nabla r^\varepsilon) q^\varepsilon _r -(g^\varepsilon+ L^\varepsilon r^\varepsilon _\delta )q^\varepsilon _r \\
&=  q_r ^{\varepsilon} \Delta r^\varepsilon  + q_{r} ^{\varepsilon} \frac{h _q}{\varepsilon} (|\nabla r^\varepsilon |^2 -1)  -(u^\varepsilon \cdot \nabla r^\varepsilon) q^\varepsilon _r  - (g^\varepsilon+ L^\varepsilon r^\varepsilon _\delta)q^\varepsilon _r 
\end{split}
\end{equation*}
for any $(x,t) \in \Omega \times (0,T_1)$. Thus we obtain
\begin{equation}
r_t^\varepsilon =  \Delta r^\varepsilon + \frac{h _q}{\varepsilon} (|\nabla r^\varepsilon |^2 -1) -u^\varepsilon \cdot \nabla r^\varepsilon - g^\varepsilon - L^\varepsilon r^\varepsilon _\delta \qquad \text{in} \ \Omega \times (0,T_1).
\label{eq4.2'}
\end{equation}
Set $M^\varepsilon := \varepsilon ^{-1} \max _{|s| \leq 1} |h_q (s)| 
+ \sup _{x\in \Omega , t\in [0,T_1)}| g^\varepsilon (x,t)| $. 
We remark that $\frac{h_q (q^\varepsilon)}{\varepsilon} \leq M^\varepsilon$ 
by $\sup _{x\in \Omega, t\in[0,T)} |\varphi ^\varepsilon (x,t)| \leq 1$.
From the definition, $r^\varepsilon _\delta >0$ in $U_T^b := \{ (x,t) \in \Omega \times (0,T_1-b) \, | \, r^\varepsilon (x,t) > 0 \}$ for $b\in (0,T_1/2)$. 
Therefore we have
\begin{equation*}
\tilde r_t^\varepsilon \leq  \Delta \tilde r^\varepsilon + \Big( \frac{h _q}{\varepsilon} \nabla \tilde r^\varepsilon -u^\varepsilon \Big) \cdot \nabla \tilde r^\varepsilon \qquad \text{in} \ U_T^b,
\end{equation*}
where $\tilde r ^\varepsilon := r^\varepsilon -M^\varepsilon t$. 
By the maximum principle, we obtain 
\begin{equation}
\max _{x\in \Omega , t\in [0,T_1-b]}  r ^\varepsilon (x,t)  \leq \max _{x\in \Omega} | r^\varepsilon (x,0)| + M^\varepsilon T_1.
\label{eq4.4'}
\end{equation}
The definition of $T_1$ implies $\lim _{b\downarrow 0} \max _{x\in \Omega , t\in [0,T_1-b]}  r ^\varepsilon (x,t) =\infty$.
This contradicts \eqref{eq4.4'} and $\varphi ^\varepsilon (x,t) <1$ for any $(x,t) \in \Omega \times [0,T)$. 
Similarly, we obtain
$\varphi ^\varepsilon (x,t) >-1$ for any $(x,t) \in \Omega \times [0,T)$.
In addition, $\max_{x\in \Omega} |r^\varepsilon (x,0)| \leq (q^\varepsilon)^{-1} (1-a) $ imply
\[
\max _{x\in \Omega, t \in [0,T] } |\varphi ^\varepsilon (x,t)| 
\leq
q^\varepsilon ((q^\varepsilon)^{-1} (1-a)+M^\varepsilon T) <1.
\]
Thus $\sup _{(x,t) \in \Omega \times [0,T)} |\varphi ^\varepsilon (x,t)| \leq 1-\delta$ 
holds for sufficiently small $\delta>0$.
\end{proof}

By Lemma \ref{prop4.1}, the standard parabolic PDE
theory shows 
\begin{proposition}\label{prop-existence}
Let $T>0$ and $\varphi ^\varepsilon $ be a smooth function on $\Omega$ with $\max_{x\in \Omega}|\varphi _0 ^\varepsilon (x)| <1$. Then there exists a unique solution $\varphi^\varepsilon$ for \eqref{ac2} with initial data $\varphi _0 ^\varepsilon$ and $\sup_{x\in \Omega, t\in [0,T)}|\varphi ^\varepsilon (x,t)| <1$ for any $t\in (0,T)$.
\end{proposition}
%

\subsection{Non-positivity of the discrepancy measure}\mbox{}\\[0.2cm]
Set $\xi_\varepsilon (x,t) := \dfrac{\varepsilon |\nabla \varphi ^{\varepsilon} (x,t)|^2}{2} - \dfrac{W (\varphi ^{\varepsilon} (x,t))}{\varepsilon}$ for the solution $\varphi^\varepsilon$ for \eqref{ac2}. One of the key lemmas of this paper is the following:
\begin{lemma}\label{lem-negative}
Assume that $|\nabla r ^{\varepsilon} (x,0)|\leq 1$ for any $x \in \Omega$. Then we have $|\nabla r ^{\varepsilon} (x,t)|\leq 1$ and $\xi _{\varepsilon} (x,t)\leq 0$ for any $(x,t) \in \Omega \times [0,T)$. Moreover $\xi _t ^{\varepsilon} $ is a non-positive measure for $t \in [0,T)$.
\end{lemma}
\begin{proof}
By \eqref{q} we have
\[ \frac{\varepsilon |\nabla \varphi ^{\varepsilon} |^2/2}{W(\varphi ^{\varepsilon}) / \varepsilon}\leq |\nabla r^\varepsilon |^2 \ \ \text{on} \ \ \Omega \times[0,T).\]
Therefore, if  $|\nabla r^\varepsilon |\leq 1$ then $\xi _\varepsilon \leq 0$ and $\xi _t ^\varepsilon$ is a non-positive measure.
Thus we only need to prove that $|\nabla r^\varepsilon |\leq 1$ on $\Omega \times [0,T)$. 

By an argument similar to that in \eqref{eq4.2'},  we obtain
\begin{equation}
r_t^\varepsilon =  \Delta r^\varepsilon + \frac{h _q}{\varepsilon} (|\nabla r^\varepsilon |^2 -1) -u^\varepsilon \cdot \nabla r^\varepsilon - g^\varepsilon - L^\varepsilon r^\varepsilon,
\label{eq4.2}
\end{equation}
where $h (q)=\sqrt{2W (q)}$ for $q \in \mathbb{R}$. We compute
\begin{equation}
\begin{split}
& \nabla ( -u^\varepsilon \cdot \nabla r^\varepsilon -g^\varepsilon - L^\varepsilon r^\varepsilon) \cdot \nabla r^\varepsilon \\ 
\leq & -\frac12 u^\varepsilon \cdot \nabla |\nabla r ^\varepsilon|^2 + |\nabla r ^\varepsilon |^2 |\nabla u^\varepsilon| + \frac12 |\nabla g^\varepsilon| (1+| \nabla r^\varepsilon|^2)- L^\varepsilon | \nabla r^\varepsilon|^2\\
 \leq & -\frac12 u^\varepsilon \cdot \nabla |\nabla r ^\varepsilon|^2 + \frac12 L^\varepsilon (1-| \nabla r^\varepsilon|^2).
\end{split}
\label{eq4.3}
\end{equation}
By \eqref{eq4.2} and \eqref{eq4.3}, we have
\begin{equation}
\begin{split}
\partial _t |\nabla r^\varepsilon |^2 \leq & \Delta |\nabla r^\varepsilon |^2 - 2|\nabla^2 r^\varepsilon |^2 +\frac{2}{\varepsilon} \nabla r^\varepsilon \cdot \nabla h _q (|\nabla r^\varepsilon |^2 -1) + \left( \frac{2h _q}{\varepsilon} \nabla r^\varepsilon-\frac{u^\varepsilon}{2} \right)\cdot \nabla |\nabla r^\varepsilon |^2\\
&+\frac12 L^\varepsilon (1-| \nabla r^\varepsilon|^2).
\end{split}
\label{max}
\end{equation}
Set $w^\varepsilon :=|\nabla r^\varepsilon|^2 -1 $. By \eqref{max} we obtain
\begin{equation}
\partial _t w^\varepsilon \leq \Delta w^\varepsilon  + \left( \frac{2h _q}{\varepsilon} \nabla r^\varepsilon-\frac{u^\varepsilon}{2} \right)\cdot \nabla w^\varepsilon +\left( \frac{2}{\varepsilon} \nabla r^\varepsilon \cdot \nabla h _q -\frac12 L^\varepsilon \right)w^\varepsilon .
\label{max2}
\end{equation}
By the assumption we have $w^\varepsilon (\cdot,0) =|\nabla r^\varepsilon (\cdot,0)|^2 -1\leq 0$ on $\Omega$. Therefore by \eqref{max2} and the maximum principle we obtain $w^\varepsilon \leq 0$ on $\Omega \times [0,T)$. Hence we have $|\nabla r^\varepsilon| \leq 1$ on $\Omega \times [0,T)$.
\end{proof}

\begin{remark}
In the case of the volume preserving MCF, that is, $u^\varepsilon \equiv 0$, $L^\varepsilon \equiv 0$, and $g^\varepsilon =g^\varepsilon (t)$ be a non-local term of $\varphi^\varepsilon$, similar estimates (including the monotonicity formula below) have been proven in \cite{takasao2017}.
\end{remark}

\begin{remark}\label{rem4.2}
To obtain the estimate for $\xi _\varepsilon$, a method of applying 
the maximum principle directly to $\xi_\varepsilon$ with some additional term 
is also well known 
(\cite{xchen1996, liu-sato-tonegawa, modica, takasao-tonegawa}) 
in the case of $g^\varepsilon \equiv 0$.
In \cite{takasao-tonegawa}, they considered the maximum principle for $\tilde \xi^\varepsilon := \dfrac{\varepsilon |\nabla \varphi ^{\varepsilon} (x,t)|^2}{2} - \dfrac{W (\varphi ^{\varepsilon} (x,t))}{\varepsilon} -\dfrac{G (\varphi ^{\varepsilon} (x,t))}{\varepsilon}$ to show the following estimate:
\begin{equation}
\dfrac{\varepsilon |\nabla \varphi ^{\varepsilon} (x,t)|^2}{2} - \dfrac{W (\varphi ^{\varepsilon} (x,t))}{\varepsilon} \leq 10 \varepsilon ^{-\beta} \qquad \text{in} \ \Omega \times [0,T],
\label{rem4.2-eq1}
\end{equation}
where $\varphi ^\varepsilon$ is a solution for \eqref{ac1}, $\beta \in (0,\frac12)$ and $G$ is a function such as $G (\varphi ^{\varepsilon}) = \varepsilon ^{\frac12} \Big( 1-\frac18 (\varphi ^\varepsilon -\alpha_1 )^2\Big)$. Clearly, \eqref{rem4.2-eq1} is weaker than \eqref{negativity}, and the key of the proof of \eqref{rem4.2-eq1} is that $\tilde \xi ^\varepsilon$ satisfies 
\begin{equation*}
\partial_t \tilde \xi^\varepsilon + u^\varepsilon \cdot \nabla \tilde \xi ^\varepsilon -\Delta \tilde \xi ^\varepsilon \leq F(\varepsilon, W',G',G'',\nabla \varphi ^\varepsilon,\nabla u^\varepsilon)
\end{equation*}
for suitable $F$ (see \cite[(4.32)]{takasao-tonegawa}). However, in the case of $g\not =0$, it is not known whether similar estimates can be obtained in this way, because $F\leq 0$ is not necessarily and the control of the term $g^\varepsilon \tilde \xi ^\varepsilon$ is more difficult than that of the term $u^\varepsilon \cdot \nabla \tilde \xi ^\varepsilon$, from the viewpoint of the maximum principle.
\end{remark}


\subsection{$L^2$-estimates of transport term and forcing term}\mbox{}\\[0.2cm]
The following estimate corresponds to the $L^2 (\mu ^\varepsilon _t )$-estimate of $f^\varepsilon$.
\begin{lemma}\label{lemma4.6}
Assume that $|\varphi^\varepsilon | <1 $ and $|\nabla r^\varepsilon| \leq 1$ in $\Omega \times [0,T)$, $p\in[2d/(d+1) ,\infty)$, and $0\leq L^\varepsilon \leq \varepsilon^{ -\gamma}$ for $\gamma >0$. Then we have
\begin{equation}
\begin{split}
& \int _{\Omega} \varepsilon \Big( (u^\varepsilon \cdot \nabla \varphi ^\varepsilon) +(g^\varepsilon +L^\varepsilon r^\varepsilon) \frac{\sqrt{2W(\varphi ^\varepsilon)}}{\varepsilon} \Big)^2 \, dx =
2\int _{\Omega} |f^\varepsilon |^2 \frac{W(\varphi ^\varepsilon)}{\varepsilon}  \, dx\\
 \leq & \Cl{const:1} (D(t) (\|u^\varepsilon (\cdot ,t) \|_{W^{1,p} (\Omega)} ^2 +\|g^\varepsilon (\cdot ,t) \|_{W^{1,p} (\Omega)} ^2)+ \varepsilon ^{1-2\gamma}),
\end{split}
\label{eq4.6}
\end{equation}
where $\Cr{const:1}=\Cr{const:1}(d,p,W,|\Omega|)>0$.
\end{lemma}

\begin{proof}
We compute
\begin{equation}
\begin{split}
&\int _{\Omega} \varepsilon \Big( (u^\varepsilon \cdot \nabla \varphi ^\varepsilon) +(g^\varepsilon +L^\varepsilon r^\varepsilon)  \frac{\sqrt{2W(\varphi ^\varepsilon)}}{\varepsilon} \Big)^2 \, dx 
= 2 \int _{\Omega} \Big( u^\varepsilon \cdot \nabla r^\varepsilon +g^\varepsilon +L^\varepsilon r^\varepsilon \Big)^2 \frac{W(\varphi ^\varepsilon)}{\varepsilon} \, dx \\
= & 2\int _{\Omega} |f^\varepsilon |^2 \frac{W(\varphi ^\varepsilon)}{\varepsilon}  \, dx \leq 6 \int _{\Omega} |u^\varepsilon |^2 \frac{W(\varphi ^\varepsilon)}{\varepsilon} \, dx
+ 6 \int _{\Omega} (g^\varepsilon )^2 \frac{W(\varphi ^\varepsilon)}{\varepsilon} \, dx
 + 6 \int _{\Omega} (L^\varepsilon r^\varepsilon )^2 \frac{W(\varphi ^\varepsilon)}{\varepsilon} \, dx,
\end{split}
\label{eq4.7}
\end{equation}
where $|\nabla r^\varepsilon | \leq 1$ and $\nabla \varphi ^\varepsilon =q^\varepsilon _r \nabla r^\varepsilon=\varepsilon ^{-1} \sqrt{2W(\varphi^\varepsilon)} \nabla r^\varepsilon$ are used.

Next we show that there exists $C>0$ such that
\begin{equation}
\int _{\Omega} (L^\varepsilon r^\varepsilon ) ^2 \frac{W(\varphi^\varepsilon)}{\varepsilon} \, dx \leq C \varepsilon ^{1-2\gamma}.
\label{eq4.8}
\end{equation}
We remark that $q^\varepsilon (r) = q(r/\varepsilon)$ and $r^\varepsilon =\varepsilon q^{-1} (\varphi^\varepsilon)$. Thus we have
\begin{equation*}
\begin{split}
\int _{\Omega} (L^\varepsilon r^\varepsilon ) ^2 \frac{W(\varphi^\varepsilon)}{\varepsilon} \, dx 
\leq \int _{\Omega} \varepsilon ^{1-2\gamma} (q^{-1} (\varphi ^\varepsilon))^2 W(\varphi^\varepsilon) \, dx\leq \varepsilon ^{1-2\gamma} \Cr{const:w} |\Omega| ,
\end{split}
\end{equation*}
where \eqref{w4} is used. Hence we obtain \eqref{eq4.8}. 

Finally we show that there exists $C>0$ such that
\begin{equation}
\int _{\Omega} |u^\varepsilon |^2 \frac{W(\varphi ^\varepsilon)}{\varepsilon} \, dx \leq CD(t) \|u^\varepsilon (\cdot ,t) \|_{W^{1,p} (\Omega)} ^2.
\label{eq4.9}
\end{equation}
Let $\{ \psi_i \} _{i} $ be a partition of unity on $\Omega$ with $\psi _i  \in C_c ^{\infty} (\Omega)$, $\diam (\spt \psi _i ) \leq 1/2$ and $\| \psi _i \| _{C^2} \leq c(d)$ for any $i$. First we consider the case of $2d/(d+1)\leq p<2$. Set $s:= p(d-1)/(d-p) $. Note that $s\geq 2$ and $p$ satisfies \eqref{assumption-p}. By \eqref{mzineq2} we have 
\begin{equation}
\begin{split}
\int _{\Omega} |u^\varepsilon |^2 \frac{W(\varphi ^\varepsilon)}{\varepsilon} \, dx 
\leq & \Big( \int _{\Omega} |u^\varepsilon |^s \, d\mu _t ^\varepsilon \Big)^{\frac{2}{s}} (2\mu _t ^\varepsilon (\Omega))^{1-\frac{2}{s}} \\
\leq & \Big( \sum _{i} C \int _{\Omega} |\psi _i u^\varepsilon |^s \, d\mu _t ^\varepsilon \Big)^{\frac{2}{s}} (2\mu _t ^\varepsilon (\Omega))^{1-\frac{2}{s}} \\
\leq & \Big( \sum _{i} Cc_{MZ} D(t) \Big( \int _{\spt \psi _i} |u^\varepsilon |^p + |\nabla u^\varepsilon |^p \, dx \Big) ^{\frac{s}{p}}\Big)^{\frac{2}{s}} (2D(t))^{1-\frac{2}{s}} \\
\leq & CD(t) \|u^\varepsilon (\cdot ,t) \|_{W^{1,p} (\Omega)} ^2.
\end{split}
\label{eq4.10}
\end{equation}

For the case of $p\geq 2$, we compute
\begin{equation}
\begin{split}
\int _{\Omega} |u^\varepsilon |^2 \frac{W(\varphi ^\varepsilon)}{\varepsilon} \, dx 
\leq & \Big( \int _{\Omega} |u^\varepsilon |^p \, d\mu _t ^\varepsilon \Big)^{\frac{2}{p}} (2\mu _t ^\varepsilon (\Omega))^{1-\frac{2}{p}} \\
\leq & \Big( \sum _{i} C \int _{\Omega} |\psi _i u^\varepsilon |^p \, d\mu _t ^\varepsilon \Big)^{\frac{2}{p}} (2D(t))^{1-\frac{2}{p}} \\
\leq & \Big( \sum _{i} Cc_{MZ} D(t) \int _{\spt \psi _i} |u^\varepsilon |^p + |u^\varepsilon |^{p-1}|\nabla u^\varepsilon| \, dx \Big)^{\frac{2}{p}} (2D(t))^{1-\frac{2}{p}} \\
\leq & CD(t) \|u^\varepsilon (\cdot ,t) \|_{W^{1,p} (\Omega)} ^2,
\end{split}
\label{eq4.11}
\end{equation}
where \eqref{mzineq} with $p=1$ is used.
By \eqref{eq4.10} and \eqref{eq4.11} we have \eqref{eq4.9}. Similarly, we have
\begin{equation}
\int _{\Omega} (g^\varepsilon )^2 \frac{W(\varphi ^\varepsilon)}{\varepsilon} \, dx \leq CD(t) \|g^\varepsilon (\cdot ,t) \|_{W^{1,p} (\Omega)} ^2.
\label{eq4.9g}
\end{equation}
Therefore by \eqref{eq4.7}, \eqref{eq4.8}, \eqref{eq4.9}, and \eqref{eq4.9g} we obtain \eqref{eq4.6}. 
\end{proof}

\begin{remark}\label{rem4.3}
The estimate \eqref{eq4.8} means that if $\|\nabla u^\varepsilon, \nabla g^\varepsilon\| _{\infty} \leq \varepsilon ^{-\gamma}$ for $\gamma \in [0,1/2)$, then the additional term $- L^\varepsilon r^\varepsilon \sqrt{2W (\varphi ^{\varepsilon})}$ vanishes as $\varepsilon \downarrow 0$ 
in the framework of the phase field method of this paper (see \eqref{vanish-L}). 
\end{remark}
\subsection{Energy estimates and monotonicity formula}\mbox{}\\[0.2cm]
Next we show the standard energy estimates and the monotonicity formula for the Allen-Cahn equation \eqref{ac2}. 
\begin{lemma}\label{lem:4.5}
Let $p\in[2d/(d+1) ,\infty)$ and $2<q<\infty$. Then there exists $\Cl{const:2} =\Cr{const:2}(d,p,q,W,|\Omega|)>0$ such that for any $0\leq t_1 <t_2<T $ we have
\begin{equation}
\begin{split}
& \sup_{t\in [t_1,t_2]} \mu _t ^\varepsilon (\Omega) +\frac{1}{2\sigma} \int _{t_1} ^{t_2} \int_\Omega\varepsilon \Big( \Delta \varphi ^\varepsilon -\frac{W'(\varphi ^\varepsilon)}{\varepsilon^2} \Big)^2 \, dxdt \\
\leq & \mu _{t_1} ^\varepsilon (\Omega) + \Cr{const:2} \Big\{ (t_2-t_1) ^{1-\frac{2}{q}}( \| u^\varepsilon \|^2 _{L^q ([t_1,t_2] ; (W^{1,p} (\Omega))^d)} + \| g^\varepsilon \|^2 _{L^q ([t_1,t_2] ; W^{1,p} (\Omega))} )\sup _{t\in [t_1,t_2]} D(t) \\
& \qquad \qquad \qquad + (t_2-t_1) \varepsilon ^{1-2\gamma }  \Big\}.
\end{split}
\label{eq4.12}
\end{equation}
\end{lemma}

\begin{proof}
By \eqref{ac3} and the integration by parts, we have
\begin{equation}
\frac{d}{dt}\mu _t ^\varepsilon (\Omega) +\frac{1}{2\sigma} \int_\Omega \varepsilon\Big( \Delta \varphi ^\varepsilon -\frac{W'(\varphi ^\varepsilon)}{\varepsilon^2} \Big)^2 \, dx
\leq \frac{1}{\sigma} \int _{\Omega} |f^\varepsilon|^2 \frac{W(\varphi^\varepsilon )}{\varepsilon} \, dx.
\label{eq4.13}
\end{equation}
Integration of \eqref{eq4.13} over $[t_1,t_2]$ with \eqref{eq4.6} gives \eqref{eq4.12}.
\end{proof}

To localize the backward heat kernel $\rho$, we fix a radially symmetric cut-off function
\[ \eta (x) \in C_c ^\infty (B_{1/2} ^d (0))  \quad \text{with} \quad \eta =1 \ \ \text{on} \ \ B_{1/4} ^d (0) , \ 0\leq \eta \leq 1, \]
and  we define $\tilde \rho _{y,s} (x,t) :=\eta (x-y) \rho _{y,s} (x,t)$.
The following estimate is the monotonicity formula for the modified equation \eqref{ac2}. 
\begin{lemma}\label{propmono}
Assume that $d\geq 2$, $T>0$, $\varphi^\varepsilon$ is a solution for \eqref{ac2} and the initial data satisfies $|\varphi ^\varepsilon_0 (x)|<1$ and $|\nabla r ^{\varepsilon} (x,0)|\leq 1$ for any $x \in \Omega$. Then
\begin{equation}
\begin{split}
\frac{d}{dt} \int _{\mathbb{R}^d} \rho _{y,s} (x,t) \, d\mu ^\varepsilon _t (x) \leq & \frac{1}{2\sigma} \int _{\mathbb{R}^d} \rho _{y,s} (x,t) |f^\varepsilon(x,t)|^2 \frac{W(\varphi ^\varepsilon(x,t))}{\varepsilon}  \, dx\\
 \leq & \frac{1}{2} \int _{\mathbb{R}^d} \rho _{y,s} (x,t) |f^\varepsilon(x,t)|^2 \, d\mu_t ^\varepsilon (x)
\end{split}
\label{monoton}
\end{equation}
and
\begin{equation}
\begin{split}
\frac{d}{dt} \int _{\mathbb{R}^d} \tilde \rho _{y,s} (x,t) \, d\mu ^\varepsilon _t (x) \leq & \frac{1}{2\sigma} \int _{\mathbb{R}^d} \tilde \rho _{y,s} (x,t) |f^\varepsilon(x,t)|^2 \frac{W(\varphi ^\varepsilon(x,t))}{\varepsilon}  \, dx\\
  & + \Cl{const:3} e^{-\frac{1}{128(s-t)}} \mu _t ^\varepsilon (B^d _{1/2} (y))
\end{split}
\label{monoton2}
\end{equation}
for any $y\in \mathbb{R}^d$, $0\leq t<s<T$ and $\varepsilon \in (0,1)$. Here $\Cr{const:3}=\Cr{const:3}(d)>0$, $\mu_t ^\varepsilon$ and $f^\varepsilon$ are extended periodically to $\mathbb{R}^d$. 
\end{lemma}

\begin{proof}
In this proof, we regard all functions and measures as periodically extended on $\mathbb{R}^d$. Set $\rho=\rho_{y,s}(x,t)$. By an argument similar to that in the proof of Proposition 2.7 in \cite{takasao2017}, we have
\begin{equation}
\begin{split}
\frac{d}{dt} \int _{\mathbb{R}^d} \rho \, d\mu_t ^\varepsilon\leq \frac{1}{2(s-t)}\int _{\mathbb{R}^d} \rho \, d\xi _t ^\varepsilon
+\frac{1}{2\sigma} \int _{\mathbb{R}^d} \rho |f^\varepsilon|^2 \frac{W(\varphi ^\varepsilon)}{\varepsilon}  \, dx.
\end{split}
\label{mono1}
\end{equation}
By Lemma \ref{lem-negative} and \eqref{mono1}, we obtain 
\begin{equation*}
\begin{split}
\frac{d}{dt} \int _{\mathbb{R}^d} \rho \, d\mu_t ^\varepsilon\leq
\frac{1}{2\sigma} \int _{\mathbb{R}^d} \rho |f^\varepsilon|^2 \frac{W(\varphi ^\varepsilon)}{\varepsilon}  \, dx
\leq \frac{1}{2} \int _{\mathbb{R}^d} \rho |f^\varepsilon|^2 \, d\mu_t ^\varepsilon.
\end{split}
\end{equation*}
Therefore we have \eqref{monoton}. In the computation \eqref{monoton} with $\tilde \rho$ instead of $\rho$, we obtain additional terms with the differentiation of $\eta$. Note that the integration of these terms are estimated by $c\mu _t ^\varepsilon (B^d _{1/2} (y)) e^{-\frac{1}{128(s-t)}} $ with $c=c(d) >0$ because $|\partial _{x_j} \rho| \leq c(j,d) e^{-\frac{1}{128(s-t)}} $ for any $x \in\Omega$ with $|x-y| >1/4$ and $j=0,1$. Therefore we obtain \eqref{monoton2}.
\end{proof}

The following estimates are given in \cite{takasao-tonegawa}. Thus we skip the proof.
\begin{lemma}\label{lem4.7}
Let $2<q<\infty$ and $p\in[\frac{2d}{d+1} ,\infty) \cap ( \frac{dq}{2(q-1)}, \infty)$. Then there exists $\Cl{const:4}=\Cr{const:4}(d,p,q)>0$ such that for any $0\leq t_1 <t_2 <s<T$ we have
\begin{equation}
\int _{t_1} ^{t_2} \int _{\mathbb{R}^d} \tilde {\rho} _{y,s} |u^\varepsilon|^2 \, d\mu_t ^\varepsilon dt \leq \Cr{const:4} (t_2 -t_1)^{\hat{p}}  \| u^\varepsilon \|^2 _{L^q ([t_1,t_2] ; (W^{1,p} (B_{1/2}^d (y)))^d)} \sup _{t\in [t_1,t_2]} D(t)
 \label{eq4.16}
\end{equation}
and
\begin{equation}
\int _{t_1} ^{t_2} \int _{\mathbb{R}^d} \tilde {\rho} _{y,s} |g^\varepsilon|^2 \, d\mu_t ^\varepsilon dt \leq \Cr{const:4} (t_2 -t_1)^{\hat{p}}  \| g^\varepsilon \|^2 _{L^q ([t_1,t_2] ; W^{1,p} (B_{1/2}^d (y)))} \sup _{t\in [t_1,t_2]} D(t),
 \label{eq4.17}
\end{equation}
where $\hat{p}$ is given by $\hat{p} = \frac{2pq-2p -dq}{pq} >0$ when $p<d$, $\hat{p} <\frac{q-2}{q}$ can be taken arbitrarily close to $\frac{q-2}{q}$ (however $c$ depends on $\hat{p}$ in addition), and $\hat{p} =\frac{q-2}{q}$ when $p>d$.
\end{lemma}

\subsection{Proof of Theorem \ref{theorem-density}}\mbox{}\\[0.2cm]
In this section we prove the upper bound of the density of $\mu _t ^\varepsilon$ via the monotonicity formula. The proof is based on \cite{liu-sato-tonegawa, takasao-tonegawa}.
\begin{lemma}\label{lemma4.8}
Assume that $2<q<\infty$ and $p\in[\frac{2d}{d+1} ,\infty) \cap ( \frac{dq}{2(q-1)}, \infty)$. Then there exist $\tilde c\geq 2$, $c'>0$ and $\epsilon_1>0$ with the following property. For $0 \leq t_1 <t_2<T$ with $t_2 -t_1 <1$, suppose $D(t_2) =\tilde c D(t_1)$ and $ D(t) < D(t_2) $ for $t_1 \leq t <t_2$. Then for any $0<\varepsilon <\epsilon_1$, we have
\begin{equation}
(t_2 -t_1)^{\hat{p}} ( \| u^\varepsilon \|^2 _{L^q ([t_1,t_2] ; (W^{1,p} (\Omega))^d)} + \| g^\varepsilon \|^2 _{L^q ([t_1,t_2] ; W^{1,p} (\Omega))} )\geq c',
\label{eq4.21}
\end{equation}
where $\hat{p} >0$ is as Lemma \ref{lem4.7}.
\end{lemma}

\begin{proof}
Set $A:= \| u^\varepsilon \|^2 _{L^q ([t_1,t_2] ; (W^{1,p} (\Omega))^d)} + \| g^\varepsilon \|^2 _{L^q ([t_1,t_2] ; W^{1,p} (\Omega))}$. Let $\tilde c \geq 2$ and assume $D(t_2) =\tilde c D(t_1)$ ($\tilde c$ will be chosen later). We consider the following three cases. First we consider the case of $D(t_2) =\mu _{t_2} ^\varepsilon (\Omega)$. By \eqref{eq4.12} we have
\begin{equation*}
\begin{split}
D(t_2) \leq & D(t_1) + \Cr{const:2} \Big\{ (t_2-t_1) ^{1-\frac{2}{q}}AD(t_2) + (t_2-t_1) \varepsilon ^{1-2\gamma }  \Big\}.
\end{split}
\end{equation*}
Therefore we obtain
\begin{equation*}
\begin{split}
D(t_1)\Big(\tilde c- \tilde c\Cr{const:2} (t_2-t_1) ^{\hat{p}}A -1 \Big) \leq D(t_1)\Big(\tilde c- \tilde c\Cr{const:2} (t_2-t_1) ^{1-\frac{2}{q}}A -1 \Big)\leq \Cr{const:2} \varepsilon ^{1-2\gamma },
\end{split}
\end{equation*}
where $\hat{p} \leq 1-\frac{2}{q}$ is used. Thus, we have \eqref{eq4.21}, for sufficiently large $\tilde c \geq 2$ and sufficiently small $\varepsilon >0$.

Next we consider the case of $D(t_2) = \lim_{n \to \infty} \frac{\mu_{t_2}^\varepsilon (B_{r_n} (y)) }{\omega _{d-1} r_n ^{d-1}} $ with $\lim _{n \to \infty} r_n \geq \frac{1}{4}$. Then  there exists $n\geq 1$ such that $r_n \geq \frac15$ and $D(t_2) -\frac{1}{100} \leq \frac{\mu_{t_2}^\varepsilon (B_{r_n} (y)) }{\omega _{d-1} r_n ^{d-1}}$. Therefore we have
\[
\frac{\omega _{d-1}}{5^{d-1}} D(t_2)- \frac{\omega _{d-1}}{5^{d-1} \cdot100} \leq \mu _{t_2}^\varepsilon (\Omega)  .
\]
Hence, by an argument similar to that in the first case, we obtain
\begin{equation*}
\begin{split}
D(t_1)\Big(\tilde c\frac{\omega _{d-1}}{5^{d-1}}- \tilde c\Cr{const:2} (t_2-t_1) ^{\hat{p}}A -1 \Big) 
\leq \Cr{const:2} \varepsilon ^{1-2\gamma } + \frac{\omega _{d-1}}{5^{d-1} \cdot100} .
\end{split}
\end{equation*}
Thus, we have \eqref{eq4.21}, for sufficiently large $\tilde c \geq 2$ and sufficiently small $\varepsilon >0$.

Finally we consider the case of $D(t_2) = \lim_{n \to \infty} \frac{\mu_{t_2} ^\varepsilon (B_{r_n} (y)) }{\omega _{d-1} r_n ^{d-1}} $ with $\lim _{n \to \infty} r_n <\frac{1}{4}$. Then  there exists $n\geq 1$ such that $0<r_n < \frac14$ and 
\begin{equation}
D(t_2) -\frac{1}{100} \leq \frac{\mu_{t_2}^\varepsilon (B_{r_n} (y)) }{\omega _{d-1} r_n ^{d-1}}.
\label{eq4.26}
\end{equation}
Set $R=r_n$ and $s=t_2 +\frac{R^2}{4}$. We compute that
\begin{equation}
\begin{split}
&\int _{\mathbb{R}^d} \tilde \rho _{y,s}(x,t_1) \, d\mu _{t_1} ^{\varepsilon }(x) \leq \frac{1}{(4\pi (s-t_1))^{\frac{d-1}{2}} } \int _{\mathbb{R}^d} e^{-\frac{|x-y|^2}{4(s-t_1)}} \, d\mu _{t_1} ^{\varepsilon } \\
= & \frac{1}{(4\pi (s-t_1))^{\frac{d-1}{2}} } \int _0 ^1 \mu _{t_1} ^{\varepsilon }( \{ x \, | \, e^{-\frac{|x-y|^2}{4(s-t_1)}}>k \}) \, dk \\
= & \frac{1}{(4\pi (s-t_1))^{\frac{d-1}{2}} } \int _0 ^1 \mu _{t_1} ^{\varepsilon}( B_{\sqrt{4(s-t_1)\log k^{-1}}} (y)) \, dk\\
\leq &  \frac{1}{(4\pi (s-t_1))^{\frac{d-1}{2}} } \int _0 ^1 D(t_1) \omega_{d-1} (\sqrt{4(s-t_1)\log k^{-1}})^{d-1}  \, dk \leq \Cl{const:5} D(t_1),
\end{split}
\label{eq4.22}
\end{equation}
where $\Cr{const:5} >0$ is depending only on $d$. By \eqref{w4} we have
\begin{equation}
\begin{split}
&\int _{t_1} ^{t_2} \int _{\mathbb{R}^d} \tilde \rho _{y,s} (L^\varepsilon r^\varepsilon ) ^2 \frac{W(\varphi^\varepsilon)}{\varepsilon} \, dxdt 
\leq  \int _{t_1} ^{t_2} \int _{\mathbb{R}^d} \varepsilon ^{1-2\gamma} \rho _{y,s} (q^{-1} (\varphi ^\varepsilon))^2 W(\varphi^\varepsilon) \, dxdt \\
\leq & 2\sqrt{\pi} \Cr{const:w} \varepsilon ^{1-2\gamma} \int _{t_1} ^{t_2}  (s- t)^{\frac12} \, dt 
\leq 2\sqrt{\pi} \Cr{const:w} \varepsilon ^{1-2\gamma} (s- t_1 )^{\frac12} (t_2 -t_1),
\end{split}
\label{eq4.23}
\end{equation}
where $\int_{\mathbb{R}^{d}} (4\pi (s-t))^{-\frac{1}{2}} \rho \, dx =1$ is used. 
From $0< R <\frac14 $ and $\eta (y-x) =1$ on $B_{\frac14} (y)$, we obtain
\begin{equation}
\begin{split}
& \int _{\mathbb{R}^d} \tilde \rho _{y,s}(x,t_2) \, d\mu _{t_2} ^{\varepsilon}=\int _{\mathbb{R}^d} \eta (y-x) \frac{1}{\pi ^{\frac{d-1}{2}} R^{d-1}} e^{-\frac{|x-y|^2}{R^2}} \, d \mu _{t_2} ^{\varepsilon }\geq \int _{B_R (y)} \frac{1}{\pi ^{\frac{d-1}{2}} R^{d-1}} e^{-\frac{|x-y|^2}{R^2}} \, d \mu _{t_2} ^{\varepsilon}\\
\geq & \int _{B_R (y)} \frac{1}{\pi ^{\frac{d-1}{2}} R^{d-1}} e^{-1} \, d \mu _{t_2} ^{\varepsilon }=\frac{1}{e \pi ^{\frac{d-1}{2}} R^{d-1}}\mu _{t_2}^{\varepsilon} (B_R (y)).
\end{split}
\label{eq4.24}
\end{equation}
By \eqref{monoton2}, \eqref{eq4.16}, \eqref{eq4.17}, \eqref{eq4.22}, \eqref{eq4.23}, and \eqref{eq4.24} we have
\begin{equation}
\begin{split}
& \frac{1}{e \pi ^{\frac{d-1}{2}}} \frac{\mu _{t_2} ^{\varepsilon} (B_R (y))}{\omega ^{d-1} R^{d-1}}  \leq \int _{\mathbb{R}^d} \tilde \rho _{y, s } (x,t_2) \, d\mu ^\varepsilon _{t_2} (x)  \\
\leq & \int _{\mathbb{R}^d} \tilde \rho _{y, s } (x,t_1) \, d\mu ^\varepsilon _{t_1} (x)  + \int _{t_1} ^{t_2} \frac{1}{2\sigma} \int _{\mathbb{R}^d} \tilde \rho _{y,s} (x,t) |f^\varepsilon(x,t)|^2 \frac{W(\varphi ^\varepsilon(x,t))}{\varepsilon}  \, dx dt \\ 
& + \Cr{const:3} \int _{t_1} ^{t_2} e^{-\frac{1}{128(s-t)}} \mu _t ^\varepsilon (B^d _{1/2} (y)) \, dt\\
\leq & \Cr{const:5} D(t_1) +  \Cr{const:w} \varepsilon ^{1-2\gamma} (s- t_1 )^{\frac12} (t_2 -t_1) + \Cr{const:3}   e^{-\frac{1}{128(s-t_1)}}  (t_2 -t_1) \sup _{t\in [t_1,t_2]} D(t) \\
& + \Cr{const:4}(t_2 -t_1)^{\hat{p}}  A \sup _{t\in [t_1,t_2]} D(t)\\
\leq & \Cr{const:5} D(t_1) +  \Cr{const:w}  \sqrt{2} \varepsilon ^{1-2\gamma}+ \Cr{const:3}(t_2 -t_1) D(t_2) + \Cr{const:4} (t_2 -t_1)^{\hat{p}}  A D(t_2),
\end{split}
\label{eq4.25}
\end{equation} 
where $s-t_1 \leq t_2 +\frac{R^2}{4} -t_1 \leq 2$ is used. By \eqref{eq4.26} and \eqref{eq4.25} we have
\begin{equation}
\begin{split}
D(t_1) \Big\{ \tilde c \Big( \frac{ 1 }{e \pi ^{\frac{d-1}{2}}} - \Cr{const:3} (t_2-t_1) - \Cr{const:4} (t_2 -t_1)^{\hat{p}} A \Big) -\Cr{const:5} \Big\}
\leq \Cr{const:w}\sqrt{2} \varepsilon ^{1-2\gamma} +  \frac{ 1 }{100 \cdot e \pi ^{\frac{d-1}{2}}} .
\end{split}
\label{eq4.27}
\end{equation} 
Thus, we have \eqref{eq4.21}, for sufficiently large $\tilde c \geq 2$ and sufficiently small $\varepsilon >0$.
\end{proof}

\bigskip

\begin{proof}[Proof of Theorem \ref{theorem-density}]
We only need to prove (2). Choose $T_b \in (0,1)$ such that 
\begin{equation}
T_b ^{\hat{p}} B \leq c',
\label{eq4.28}
\end{equation}
where $B:= \| u^\varepsilon \|^2 _{L^q ([0,T] ; (W^{1,p} (\Omega))^d)} + \| g^\varepsilon \|^2 _{L^q ([0,T] ; W^{1,p} (\Omega))}$. Note that $T_b$ depends only on $d,p,q$, and $B$, by  Lemma \ref{lemma4.8}. Define 
\[ 
D_1 := D_0 \tilde{c} ^{[T/T_b] +1}, 
\]
where $D_1$ depends only on $d,p,q,B, T, D_0$ and $D_1 \geq 2 D_0$ by $\tilde c \geq 2$. Assume that $\varepsilon \in (0,\epsilon _1)$. Note that we only need to check that
\begin{equation}
D(t) \leq D_0 \tilde c ^{[t/T_b] +1}, \qquad t \in [0,T].
\label{eq4.29}
\end{equation}
Suppose that there exists $t' \in (0,T]$ such that $D(t') > D_0 \tilde c ^{[t'/T_b] +1}$. Then there exists $\tau \in (0,T)$ such that $D(t) \leq D_0 \tilde c ^{[t/T_b] +1} \leq D_1$ for any $t \in [0,\tau ]$ and $D(\tau ) = D_0 \tilde c ^{[\tau /T_b] +1}$. Assume $\tau \in (0,T_b)$. Then we have $D(\tau) = \tilde c D_0$ and $\sup _{t \in [0,\tau]} D(t) \leq \tilde c D_0$. Thus \eqref{eq4.21} implies $\tau ^{\hat{p}} B \geq c'$, where we used Lemma \ref{lemma4.8} with $t_1=0$ and $t_2=\tau$. But this contradicts $\tau <T_b$ and \eqref{eq4.28}. Therefore we have $\tau \geq T_b$. If $\tau \in [T_b ,2T_b)$, then $D(\tau) =D_0 \tilde c^2$ and $D(t)\leq D_0 \tilde c$ for any $t\in[0,T_b)$. Hence there exists $\tau' \in [T_b , \tau)$ such that $D(\tau') =\tilde c D_0$ and $\tau-\tau'<T_b$. By Lemma \ref{lemma4.8} with $t_1=\tau'$ and $t_2=\tau$, we have $(\tau-\tau')^{\hat{p}} B \geq c'$. But this contradicts $\tau- \tau' <T_b$ and \eqref{eq4.28} again. Repeating this argument, we obtain $\tau=T$ and \eqref{eq4.29}.
\end{proof}

\subsection{Proof of Theorem \ref{theorem-existence}}\mbox{}\\[0.2cm]
Finally, we show the existence theorem for \eqref{mcf} in the sense of $L^2$-flow. 
We can easily show the existence of a $L^2$-flow 
by the result of Theorem 3.1 in \cite{mugnai-roger2011}(see Theorem \ref{mr}). 
However, we need to prove $v=h+(u\cdot \nu) \nu +g \nu$ in addition.

\begin{proof}[Proof of Theorem \ref{theorem-existence}]
Fix $T>0$. Because $T_i >T$ for sufficiently large $i\geq 1$, so we may assume $T_i >T$ for any $i\geq 1$. 
By a standard argument similar to that in \cite[Proposition 8.3]{takasao-tonegawa} we obtain (2). 

Set $G^\varepsilon (x,t) := f^\varepsilon (x,t) \sqrt{2W(\varphi ^ \varepsilon (x,t))}$. 
Then Lemma 4.3 and \eqref{density-ineq} imply
\begin{equation*}
\begin{split}
& \int_0 ^T \int _{\Omega} \frac{1}{\varepsilon} |G^\varepsilon |^2  \, dx dt =
2\int _0 ^T \int _{\Omega} |f^\varepsilon |^2 \frac{W(\varphi ^\varepsilon)}{\varepsilon}  \, dx dt\\
 \leq & \Cr{const:1} (D(T) (\|u^\varepsilon \|_{L^2 ((0, T) : (W^{1,p} (\Omega))^d )} ^2 +\|g^\varepsilon \|_{L^2 ( (0,T) : W^{1,p} (\Omega) )} ^2)+ \varepsilon ^{1-2\gamma} T), \qquad \varepsilon \in (0, \epsilon).
\end{split}
\end{equation*}
Note that the right hand side is uniformly bounded, regarding $\varepsilon \in (0,\epsilon)$. In addition, we have $\mu_0 ^\varepsilon (\Omega) \leq D_0 $ for any $\varepsilon>0$. Therefore $\mu _t ^\varepsilon$ and $\varphi^\varepsilon$ satisfy all the assumptions of Theorem \ref{mr}. Theorem \ref{mr} implies (1) and 
there exist $v, \vec{G} \in L^2 _{loc} (0,\infty ; (L^2 (\mu _t))^d)$ such that $\{ \mu_t \}_{t \in [0,\infty)}$ is a $L^2$-flow
$v=h+\vec{G}$ with \eqref{conv-v}, by taking a subsequence $\varepsilon \to 0$. 
Here $\vec{G}$ satisfies 
\begin{equation}
\begin{split}
\lim _{\varepsilon \to 0} \frac{1}{\sigma}\int _{\Omega \times (0,T)} -G^\varepsilon\nabla \varphi^\varepsilon \cdot \Phi  \, dxdt = \int _{\Omega \times (0,T)} \vec{G} \cdot \Phi  \, d\mu
\end{split}
\label{conv-G}
\end{equation}
for any $\Phi \in C_c (\Omega \times[0,T);\mathbb{R}^d)$. We remark that
\begin{equation}
\begin{split}
\frac{1}{\sigma}\int _{\Omega \times (0,T)} -G^\varepsilon\nabla \varphi^\varepsilon \cdot \Phi  \, dxdt
=& \int _0 ^T \int _{\Omega \cap \{ |\nabla \varphi ^\varepsilon (\cdot ,t)| \not=0 \}} \Big( u^\varepsilon \cdot \frac{\nabla \varphi^\varepsilon}{|\nabla \varphi^\varepsilon|} \Big) \Big( \frac{\nabla \varphi^\varepsilon}{|\nabla \varphi^\varepsilon|} \cdot \Phi\Big) \, d\tilde  \mu ^\varepsilon _t dt\\
&+ \int _0 ^T \int _{\Omega \cap \{ |\nabla \varphi ^\varepsilon (\cdot,t) |\not =0 \}} g^\varepsilon \frac{\nabla \varphi^\varepsilon}{|\nabla \varphi^\varepsilon|} \cdot \Phi  \, d\hat \mu _t ^\varepsilon dt\\
&+\frac{1}{\sigma} \int _0 ^T \int _\Omega L^\varepsilon r^\varepsilon 
\sqrt{2W(\varphi ^\varepsilon)} \nabla \varphi ^\varepsilon \cdot \Phi \, dxdt,
\end{split}
\label{conv-G2}
\end{equation}
where $d \tilde \mu _t ^\varepsilon := \frac{\varepsilon}{\sigma} |\nabla \varphi ^\varepsilon| ^2 dx $ and 
$d\hat \mu _t ^\varepsilon := \frac{1}{\sigma} \sqrt{2W (\varphi ^\varepsilon)} |\nabla \varphi ^\varepsilon| dx$.
We compute the third term of the right hand side. We have
\begin{equation*}
\begin{split}
&\Big| \frac{1}{\sigma} \int _0 ^T \int _\Omega L^\varepsilon r^\varepsilon 
\sqrt{2W(\varphi ^\varepsilon)} \nabla \varphi ^\varepsilon \cdot \Phi \, dxdt \Big| \\
\leq &\frac{1}{\sigma} \|\Phi\|_\infty
\Big( \int_0 ^T \int _\Omega (L^\varepsilon r^\varepsilon)^2 \frac{2W}{\varepsilon} \, dxdt\Big)^{\frac12}
\Big( \int_0 ^T \int _\Omega \varepsilon |\nabla \varphi ^\varepsilon |^2 \, dxdt\Big)^{\frac12}\\
\leq &\frac{1}{\sigma} \|\Phi\|_\infty
\Big( \int_0 ^T \int _\Omega (L^\varepsilon r^\varepsilon)^2 \frac{2W}{\varepsilon} \, dxdt\Big)^{\frac12}
\Big( \int_0 ^T 2D(t) \, dt\Big)^{\frac12}.
\end{split}
\end{equation*}
Hence, \eqref{density-ineq} and \eqref{eq4.8} imply 
\begin{equation}
\lim _{\varepsilon \to 0}\frac{1}{\sigma} \int _0 ^T \int _\Omega L^\varepsilon r^\varepsilon 
\sqrt{2W(\varphi ^\varepsilon)} \nabla \varphi ^\varepsilon \cdot \Phi \, dxdt=0.
\label{vanish-L}
\end{equation}

Now we show (3). The estimates \eqref{density-ineq} and \eqref{eq4.12} give
\[ \sup _{\varepsilon \in (0,\epsilon)} \int _{0} ^{T} \int_\Omega \varepsilon \Big( \Delta \varphi ^\varepsilon -\frac{W'(\varphi ^\varepsilon)}{\varepsilon^2} \Big)^2 \, dxdt <\infty \]
for any $\varepsilon \in (0,\epsilon)$. Hence Fatou's lemma implies
\[
\liminf _{\varepsilon \to 0} \int_\Omega \varepsilon \Big( \Delta \varphi ^\varepsilon (x,t) -\frac{W'(\varphi ^\varepsilon (x,t))}{\varepsilon^2} \Big)^2 \, dx <\infty, \quad \text{a.e.} \ t\in (0,T).
\]
Therefore, by Theorem \ref{rs}, $\xi _t ^\varepsilon \to 0$ a.e. $t$. Thus we obtain (3).

Next we show (4). Fix $\delta > 0$ and $ i \geq 1$ such that $\| u^{\varepsilon_i} - u \|^2 _{L^q ([0,T] ; (W^{1,p} (\Omega))^d)} <\delta $. Set $\hat u:= u^{\varepsilon_i}$. For any $\Phi \in C_c (\Omega \times[0,T);\mathbb{R}^d)$ we have
\begin{equation}
\begin{split}
&\Big|  \int _{\Omega \times (0,\infty)} u \cdot \Phi  \, d\mu - \frac{1}{\sigma}\int _{\Omega \times (0,\infty)} u^\varepsilon \cdot \Phi  \, \varepsilon |\nabla \varphi^\varepsilon|^2dxdt \Big| \\
\leq &  \|\Phi \|_\infty  \int _{\Omega \times (0,T)} |u - \hat u| \, d\mu + \Big| \int _{\Omega \times (0,\infty)}  \hat u \cdot \Phi \, d\mu - \frac{1}{\sigma}\int _{\Omega \times (0,T)} u^\varepsilon \cdot \Phi  \, \varepsilon |\nabla \varphi^\varepsilon|^2dxdt \Big|   \\
\leq & C\delta + \Big| \int _{\Omega \times (0,\infty)}  \hat u \cdot \Phi \, d\mu - \frac{1}{\sigma}\int _{\Omega \times (0,T)} \hat u \cdot \Phi  \, \varepsilon |\nabla \varphi^\varepsilon|^2dxdt \Big| \\
& +\Big| \frac{1}{\sigma}\int _{\Omega \times (0,T)} \hat u \cdot \Phi  \, \varepsilon |\nabla \varphi^\varepsilon|^2dxdt - \frac{1}{\sigma}\int _{\Omega \times (0,T)} u^\varepsilon \cdot \Phi  \, \varepsilon |\nabla \varphi^\varepsilon|^2dxdt \Big|  \\
=:& C\delta +I_1 +I_2,
\end{split}
\label{eq4.31}
\end{equation}
where \eqref{mzineq2} is used and $C>0$ depends only on $d,p,q,D(T), \|\Phi \|_\infty$. By $\xi _t ^\varepsilon \to 0$ a.e. $t$, $d \tilde \mu _t ^\varepsilon := \frac{\varepsilon}{\sigma} |\nabla \varphi ^\varepsilon| ^2 dx \to d\mu _t$ a.e. $t$. Thus $I_1 \to 0$ as $\varepsilon \to 0$. Moreover, for sufficiently small $\varepsilon >0$, the Cauchy-Schwarz inequality gives $|I_2| \leq C \delta$, where $C>0$ depends only on $D(T), \|\Phi \|_\infty$. Hence we obtain (4). 

Next we prove (5).  
First we show that $\hat \mu_t ^\varepsilon \to \mu _t $ as Radon measures for a.e. $t$. 
We compute
\[
\left| \frac{\varepsilon |\nabla \varphi ^\varepsilon |^2 }{2} + \frac{W(\varphi ^\varepsilon)}{\varepsilon} 
-\sqrt{2W (\varphi ^\varepsilon)} |\nabla \varphi ^\varepsilon| \right|
\leq \left( 
\sqrt{ \frac{\varepsilon |\nabla \varphi ^\varepsilon |^2 }{2} } -
\sqrt{ \frac{W(\varphi ^\varepsilon)}{\varepsilon} }
\right)^2
\leq \left| \frac{\varepsilon |\nabla \varphi ^\varepsilon |^2 }{2} - \frac{W(\varphi ^\varepsilon)}{\varepsilon} \right|.
\]
Therefore  $\xi _t ^\varepsilon \to 0$ implies 
$\hat \mu _t ^\varepsilon  \to \mu _t$ a.e. $t$. 
By \eqref{density-ineq} and \eqref{mzineq2} we have
\begin{equation}
\sup _{\varepsilon \in (0,\epsilon) }\int _0 ^T \int _{\Omega} |g^\varepsilon| ^2 \, d \hat \mu _t ^\varepsilon < \infty.
\label{eq4.35}
\end{equation}
Hence there exists a vector valued function $\tilde{g}$ such that
\begin{equation}
\begin{split}
\lim _{\varepsilon \to 0} \int _0 ^T \int _{\Omega \cap \{ |\nabla \varphi ^\varepsilon (\cdot,t) |\not =0 \}} g^\varepsilon \frac{\nabla \varphi^\varepsilon}{|\nabla \varphi^\varepsilon|} \cdot \Phi  \, d\hat \mu _t ^\varepsilon dt 
= \int _{\Omega \times (0,T)} \tilde g \cdot \Phi \, d\mu
\end{split}
\label{eq4.38}
\end{equation}
for any  $\Phi \in C_c (\Omega \times[0,T);\mathbb{R}^d)$ (see \cite[Theorem 4.4.2]{hutchinson}). 
Thus we obtain (5).

Finally we show (6). By \eqref{conv-G}, \eqref{conv-G2}, \eqref{vanish-L}, and \eqref{eq4.38},
we only need to prove \eqref{eq:3.8}, $\spt \tilde g \subset \partial ^{\ast} \{ \psi =1 \}$, and
\begin{equation}
\lim _{\varepsilon \to 0} \int _0 ^T \int _{\Omega \cap \{ |\nabla \varphi ^\varepsilon (\cdot ,t)| \not=0 \}} \Big( u^\varepsilon \cdot \frac{\nabla \varphi^\varepsilon}{|\nabla \varphi^\varepsilon|} \Big) \Big( \frac{\nabla \varphi^\varepsilon}{|\nabla \varphi^\varepsilon|} \cdot \Phi\Big) \, d\tilde  \mu ^\varepsilon _t dt= \int _0 ^T \int _{\Omega} (\text{Id} -T_{x} \mu _t ) u \cdot \Phi \, d\mu
\label{eq4.30}
\end{equation}
for any  $\Phi \in C_c (\Omega \times[0,T);\mathbb{R}^d)$. 
Set $\nu ^\varepsilon :=  \frac{\nabla \varphi^\varepsilon}{|\nabla \varphi^\varepsilon|}$. We compute
\begin{equation*}
\begin{split}
&\int _0 ^T \int _{\Omega \cap \{ |\nabla \varphi ^\varepsilon (\cdot ,t)| \not=0 \}} (u^\varepsilon \cdot \nu ^\varepsilon) (\nu ^\varepsilon \cdot \Phi) \, d\tilde  \mu ^\varepsilon _t dt\\
=& \int _0 ^T \int _{\Omega \cap \{ |\nabla \varphi ^\varepsilon (\cdot ,t)| \not=0 \}} (u^\varepsilon -(\text{Id} -\nu ^\varepsilon \otimes \nu ^\varepsilon) u ^\varepsilon ) \cdot \Phi \, d\tilde  \mu ^\varepsilon _t dt.
\end{split}
\end{equation*}
Note that by the definition of the varifold and integrality of $\mu _t$, $\int (\text{Id} -\nu ^\varepsilon \otimes \nu ^\varepsilon) \Psi \, d\mu_t ^\varepsilon  \to \int T_{x} \mu _t \Psi \,  d\mu_t $ for any $\Psi \in C_c (\Omega \times[0,T);\mathbb{R}^d)$. By using this and an argument  similar to \eqref{eq4.31}, we have \eqref{eq4.30}. 

Set $k(s) := \int _0 ^s \sqrt{2W(\tau)} \, d\tau$.
Recall that $\psi =\lim_{\varepsilon \to 0} \frac{1}{2} (\varphi ^\varepsilon +1) $, $\varphi ^\varepsilon \to \pm 1$ a.e. on $\Omega \times (0,\infty)$, and
\begin{equation} 
\lim_{\varepsilon \to 0} k(\varphi ^\varepsilon )=\lim_{\varepsilon \to 0} \int_0 ^{\varphi^\varepsilon} \sqrt{2W(s)} \, ds=\sigma\Big(\psi-\frac{1}{2}\Big) \quad\text{ a.e. on } \ \Omega \times (0,\infty).
\label{eq4.36}
\end{equation}
By \eqref{eq4.36}, for any $\Phi \in C_c ^1 (\Omega ;\mathbb{R}^d)$ and $t\geq 0$, we have
\begin{equation} 
\lim _{\varepsilon \to 0} \int_{\mathbb{R}^d} \div \Phi k(\varphi ^\varepsilon ) \, dx = \int_{\mathbb{R}^d} \div \Phi \sigma\Big(\psi-\frac{1}{2}\Big) \, dx
=-\sigma \int_{\mathbb{R}^d} \Phi \cdot \nu \, d\| \nabla \psi (\cdot,t) \|,
\label{eq4.37}
\end{equation}
where $\nu (\cdot,t)$ is the inner unit normal vector of $\{ \psi(\cdot,t)=1 \}$ on $\partial ^\ast \{ \psi(\cdot,t)=1\}$. 

Fix $\delta > 0$ and $ i \geq 1$ such that $\| g^{\varepsilon_i} - g \|^2 _{L^q ([0,T] ; W^{1,p} (\Omega))} <\delta $. 
Set $\hat g :=g^{\varepsilon_i}$.
For any $\Phi \in C_c ^1 (\Omega \times[0,T);\mathbb{R}^d)$ we have
\begin{equation}
\begin{split}
&\int _{\Omega \times (0,T)} \tilde g \cdot \Phi  \, d\mu=\lim _{\varepsilon \to 0} \frac{1}{\sigma}\int _{\Omega \times (0,T)} g^\varepsilon \sqrt{2W(\varphi^\varepsilon)}\nabla \varphi^\varepsilon \cdot \Phi  \, dxdt \\
=& \lim _{\varepsilon \to 0} \frac{1}{\sigma}\int _{\Omega \times (0,T)} g^\varepsilon \nabla k(\varphi^\varepsilon) \cdot \Phi  \, dxdt
=-\lim _{\varepsilon \to 0} \frac{1}{\sigma}\int _{\Omega \times (0,T)} k(\varphi^\varepsilon)  \div (g^\varepsilon \Phi) \, dxdt.
\end{split}
\label{req4.32}
\end{equation}
By \eqref{eq4.37}, the Radon-Nikodym theorem, we have
\begin{equation}
\begin{split}
-\lim _{\varepsilon \to 0} \frac{1}{\sigma}\int _{\Omega \times (0,T)} k(\varphi^\varepsilon)  \div (\hat g \Phi) \, dxdt
=\int _{\Omega\times (0,T)} \hat g \frac{1}{\theta} \nu  \cdot \Phi \, d\mu
\end{split}
\label{limitg}
\end{equation}
for any $\Phi \in C_c ^1 (\Omega \times[0,T);\mathbb{R}^d)$. 
Here $\theta : \spt \mu \to \mathbb{N}$ is defined by 
\[
  \theta = \begin{cases}
   \Big( \frac{d\|\nabla\psi(\cdot,t)\|}{d\mu_t} \Big)^{-1} & \text{if} \ (x,t) \in \partial ^\ast \{ \psi = 1 \} ,\\
   \qquad \infty & \text{otherwise},
  \end{cases}
\]
where $\frac{1}{\theta} =0$ if $\theta=\infty$, and $ \frac{d\|\nabla\psi(\cdot,t)\|}{d\mu_t}$ is the Radon-Nikodym Derivative.
We compute
\begin{equation*}
\begin{split}
&\Big| \int _{\Omega\times (0,T)} g \nu  \cdot \Phi \, d\|\nabla \psi (\cdot,t) \| dt
-\frac{1}{\sigma}\int _{\Omega \times (0,T)} g^\varepsilon \sqrt{2W(\varphi^\varepsilon)}\nabla \varphi^\varepsilon \cdot \Phi  \, dxdt \Big|\\
=&\Big| \int _{\Omega\times (0,T)} g \frac{1}{\theta} \nu  \cdot \Phi \, d\mu 
-\frac{1}{\sigma}\int _{\Omega \times (0,T)} g^\varepsilon \sqrt{2W(\varphi^\varepsilon)}\nabla \varphi^\varepsilon \cdot \Phi  \, dxdt \Big|\\
\leq & \|\Phi \|_\infty \int _{\Omega\times (0,T)} |g-\hat g| \, d\mu +
\Big| \int _{\Omega\times (0,T)} \hat g \frac{1}{\theta} \nu  \cdot \Phi \, d\mu -\frac{1}{\sigma}\int _{\Omega \times (0,T)} g^\varepsilon \sqrt{2W(\varphi^\varepsilon)}\nabla \varphi^\varepsilon \cdot \Phi  \, dxdt \Big| \\
\leq & \|\Phi \|_\infty \int _{\Omega\times (0,T)} |g-\hat g| \, d\mu +
\Big| \int _{\Omega\times (0,T)} \hat g \frac{1}{\theta} \nu  \cdot \Phi \, d\mu + \frac{1}{\sigma}\int _{\Omega \times (0,T)} k(\varphi ^\varepsilon ) \div (\hat g \Phi )  \, dxdt \Big| \\
&+\Big| - \frac{1}{\sigma}\int _{\Omega \times (0,T)} k(\varphi ^\varepsilon ) \div (\hat g \Phi )  \, dxdt -\frac{1}{\sigma}\int _{\Omega \times (0,T)} g^\varepsilon \sqrt{2W(\varphi^\varepsilon)}\nabla \varphi^\varepsilon \cdot \Phi  \, dxdt \Big|\\
=: & J_1 + J_2 + J_3.
\end{split}
\end{equation*}
By \eqref{mzineq2} we have $J_1 \leq C \delta$ and \eqref{limitg} implies $J_2 \to 0$ as $\varepsilon \to 0$. By \eqref{mzineq2} and the integration by parts, we have
\[
J_3 \leq C \| g^{\varepsilon_i} - \hat g \|^2 _{L^q ([0,T] ; W^{1,p} (\Omega))} \leq C(\delta + \| g^{\varepsilon_i} - g \|^2 _{L^q ([0,T] ; W^{1,p} (\Omega))}),
\]
where $C>0$ depends only on $d,p,q,D(T), \|\Phi \|_\infty$. Therefore we obtain
\begin{equation}
\int _{\Omega\times (0,T)} g \nu  \cdot \Phi \, d\|\nabla \psi (\cdot,t) \| dt
=
\int _{\Omega\times (0,T)} g \frac{1}{\theta} \nu  \cdot \Phi \, d\mu 
= \lim _{\varepsilon \to 0}\frac{1}{\sigma}\int _{\Omega \times (0,T)} g^\varepsilon \sqrt{2W(\varphi^\varepsilon)}\nabla \varphi^\varepsilon \cdot \Phi  \, dxdt.
\label{eq4.34}
\end{equation}
By \eqref{req4.32} and \eqref{eq4.34} we have \eqref{eq:3.8} and  $\spt \tilde g \subset \partial ^{\ast} \{ \psi =1 \}$.
\end{proof}
\section{Appendix}
\subsection{Meyers-Ziemer inequality}\mbox{}\\[0.2cm]
Let $\mu$ be a Radon measure on $\mathbb{R}^d$ and $f :\mathbb{R}^d \to \mathbb{R}$ be a given function. To define $\mu$-measurable $f$ as a trace function, we use the following inequality:
\begin{theorem}[Meyers-Ziemer inequality]\label{thmmz}
For a Radon measure $\mu$ on $\mathbb{R}^d$ with \\$D= \sup _{r>0 ,x\in \mathbb{R}^d}\frac{\mu (B_r ^d (x))}{\omega _{d-1} r^{d-1}}$ and $1\leq p<d$,
\begin{equation}
\int _{\mathbb{R}^d} |f| ^{\frac{p(d-1)}{d-p}} \, d\mu \leq c_{MZ} D\Big( \int _{\mathbb{R}^d} |\nabla f |^p \, dx \Big)^{\frac{d-1}{d-p}}
\label{mzineq}
\end{equation}
for $f \in C^1 _c (\mathbb{R}^d)$. Here $c_{MZ} = c_{MZ}(d,p)$. See \cite{meyers-Ziemmer1977} and \cite{ziemer1989} for $p=1$. 
\end{theorem}

Set $\mu _t := \lim_{\varepsilon \downarrow 0} \mu _t ^\varepsilon $ and $D_T := \sup _{t\in [0,T), r>0 ,x\in \mathbb{R}^d}\frac{\mu_t (B_r ^d (x))}{\omega _{d-1} r^{d-1}}$. Note that, to make sense of the Brakke's inequality or the convergences (4)--(6) in Theorem \ref{theorem-existence}, we only need to define the transport term and forcing term as functions in $L^2 _{loc} (\mu _t \times dt)$. By H{\"o}lder inequality and \eqref{mzineq} we have
\begin{equation}
\begin{split}
\int _{\mathbb{R}^d} |f| ^{2} \, d\mu_t \leq &
\left( \int _{\mathbb{R}^d} |f| ^{\frac{p(d-1)}{d-p}} \, d\mu_t \right)^{\frac{2(d-p)}{p(d-1)}} (\mu _t (\spt f )) ^{\frac{pd+p-2d}{p(d-1)}} \\
\leq & (c_{MZ} D_T )^{\frac{2(d-p)}{p(d-1)}} \Big( \int _{\mathbb{R}^d} |\nabla f |^p \, dx \Big)^{\frac{2}{p}} (\mu _t (\spt f )) ^{\frac{pd+p-2d}{p(d-1)}}
\end{split}
\label{mzineq2}
\end{equation}
for any $f \in C_c ^1 (\mathbb{R}^d)$. To justify \eqref{mzineq2}, we need $\frac{p(d-1)}{d-p} \geq 2$. So we need to assume
\begin{equation}
p\geq \frac{2d}{d+1}
\label{assumption-p}
\end{equation}
for \eqref{mzineq2}.

\subsection{Existence theorem for $L^2$-flow}\mbox{}\\[0.2cm]
Let $U\subset \mathbb{R}^d$ be an open set, $\varphi ^{\varepsilon} \in C^2 (U)$ for $\varepsilon \in (0,1)$ and $\{ \varepsilon _i \}_{i=1} ^\infty$ be a positive sequence with $\varepsilon _i \to 0$. Define $\mu ^\varepsilon  (\phi):=\frac{1}{\sigma}\int _U \phi \Big( \frac{\varepsilon |\nabla \varphi ^\varepsilon|^2}{2}+\frac{W(\varphi^\varepsilon)}{\varepsilon}\Big) \, dx$ and $\xi ^\varepsilon  (\phi):=\frac{1}{\sigma}\int _U \phi \Big( \frac{\varepsilon |\nabla \varphi^\varepsilon|^2}{2}-\frac{W(\varphi^\varepsilon)}{\varepsilon}\Big) \, dx$, where $\sigma:= \int _{-1} ^1 \sqrt{2W(s)} \, ds$. The following theorem is useful for showing the vanishing of the discrepancy measure and the integrality of the limit measure:
\begin{theorem}[\cite{roger-schatzle}]\label{rs}
Assume that $d=2,3$ and
\[ \liminf_{i\to \infty}  \mu ^{\varepsilon_i} (U) <\infty, \quad \liminf_{i\to \infty}  \int _U \varepsilon_i \Big( \Delta \varphi ^{\varepsilon_i} -\frac{W'(\varphi^{\varepsilon_i} )}{\varepsilon_i ^2}\Big)^2 \, dx<\infty \] and 
\[ \mu^{\varepsilon_i} \to \mu \quad \text{as Radon measures.}\]
Then the following hold:
\begin{enumerate}
\item $ |\xi ^{\varepsilon_i} | \to 0$ as Radon measures.
\item $\mu$ is $(d-1)$-integral.
\item $\int _{U} |h|^2 \, d\mu \leq \frac{1}{\sigma} \liminf_{i\to \infty} \int _U \varepsilon_i \Big( \Delta \varphi^{\varepsilon_i} -\frac{W'(\varphi ^{\varepsilon_i} )}{\varepsilon_i ^2}\Big)^2 \, dx$, where $h$ is the generalized mean curvature vector of $\mu$.
\end{enumerate}
\end{theorem}
The following theorem is also useful for prove the existence of the weak solutions for the MCF with forcing term, in the sense of $L^2$-flow.
\begin{theorem}[Theorem 3.1 in \cite{mugnai-roger2011}]\label{mr}
Let $d=2,3$ and $\varphi^{\varepsilon}$ be a solution for the following equation:
\begin{equation}
\left\{ 
\begin{array}{ll}
\varepsilon \varphi ^{\varepsilon} _t =\varepsilon \Delta \varphi ^{\varepsilon} -\dfrac{W' (\varphi^{\varepsilon})}{\varepsilon }+G^{\varepsilon} ,& (x,t)\in \Omega \times (0,\infty).  \\
\varphi ^{\varepsilon} (x,0) = \varphi_0 ^{\varepsilon} (x) ,  &x\in \Omega.
\end{array} \right.
\label{acmr}
\end{equation}
We assume that there exists $\tilde \epsilon >0$ such that
\begin{equation*}
\sup_{\varepsilon \in (0,\tilde\epsilon)} \Big( \mu _0 ^\varepsilon (\Omega) + \int _{\Omega \times (0,T)} \frac{1}{\varepsilon} (G^\varepsilon)^2 \, dxdt\Big) <\infty
\end{equation*}
for any $T>0$. Then there exits a subsequence $\varepsilon \to0$ such that the following hold:
\begin{enumerate}
\item There exists a family of $(d-1)$-integral Radon measures $\{ \mu _t \}_{t\in [0,\infty)}$ on $\Omega$ such that
\begin{enumerate}
\item $\mu ^\varepsilon \to \mu$ as Radon measures on $\Omega \times [0,\infty)$, where $d\mu =d\mu_t dt$.
\item $\mu ^\varepsilon _t \to \mu _t$ as Radon measures on $\Omega$ for all $t \in [0,\infty)$.
\end{enumerate}
\item There exists $\vec{G} \in L^2 _{loc} (0,\infty ; (L^2 (\mu _t))^d)$ such that
\begin{equation*}
\begin{split}
\lim _{\varepsilon \to 0} \frac{1}{\sigma}\int _{\Omega \times (0,\infty)} -G^\varepsilon\nabla \varphi^\varepsilon \cdot \Phi  \, dxdt = \int _{\Omega \times (0,\infty)} \vec{G} \cdot \Phi  \, d\mu
\end{split}
\end{equation*}
for any $\Phi \in C_c (\Omega \times[0,\infty);\mathbb{R}^d)$.
\item $\{ \mu _t \}_{t\in (0,\infty)}$ is an $L^2$-flow with a generalized velocity vector $v=h+\vec{G}$ and
\begin{equation*}
\lim _{\varepsilon \to 0} \int _{\Omega \times (0,\infty)}  v^\varepsilon  \cdot \Phi \, d\mu ^\varepsilon = \int _{\Omega \times (0,\infty)} v \cdot \Phi \, d\mu  
\end{equation*}
for any $\Phi \in C_c (\Omega \times[0,\infty);\mathbb{R}^d)$, where $h$ is the generalized mean curvature vector of $\mu_t$ and
\[
  v^\varepsilon = \begin{cases}
    \frac{-\varphi^\varepsilon _t}{|\nabla\varphi^\varepsilon|} \frac{\nabla\varphi^\varepsilon}{|\nabla\varphi^\varepsilon|} & \text{if} \ |\nabla\varphi^\varepsilon|\not=0 ,\\
    \qquad 0 & \text{otherwise}.
  \end{cases}
\]
\end{enumerate}
\end{theorem}
\begin{remark}

\noindent
\begin{enumerate}
\item The assumption for $d$ comes from Theorem \ref{rs}. 
\item The boundary conditions of \eqref{acmr} of the original theorem are Neumann conditions. However, we may also obtain same results for periodic boundary conditions, with minor modification of the proof (see \cite[Remark 2.3]{mugnai-roger2008}). 
\end{enumerate}
\end{remark}

\bigskip

\noindent
\textbf{Acknowledgments}

\bigskip
\noindent
This work was supported by JSPS KAKENHI
Grant Numbers JP16K17622, JP18H03670, 
and JSPS Leading Initiative for Excellent Young Researchers(LEADER) operated by Funds for the Development of Human Resources in Science and Technology.


\begin{thebibliography}{99}
\bibitem{allard}
W. Allard,
\textit{On the first variation of a varifold},
Ann. of Math. (2) {\bf 95} (1972), 417--491.
\bibitem{allen-cahn}
S. M. Allen and J. W. Cahn, 
\textit{A macroscopic theory for antiphase boundary motion and its application to
antiphase domain coarsening},
Acta. Metal. {\bf 27} (1979), 1085--1095.
\bibitem{almgren-taylor-wang}
F. J. Almgren, J. E. Taylor and L.-H. Wang, 
\textit{Curvature-driven flows: a variational approach}, 
SIAM J. Control Optim. {\bf 31} (1993), 387--438.
\bibitem{bertini-butta-pisante}
L. Bertini, P. Butt\`a and A. Pisante,
\textit{Stochastic Allen-Cahn approximation of the mean curvature flow: large deviations upper bound},
Arch. Ration. Mech. Anal. {\bf 224} (2017), 659--707.
\bibitem{brakke}
K. A. Brakke,
\textit{The motion of a surface by its mean curvature},
Princeton University Press, Princeton, N.J., (1978).
\bibitem{bronsard-kohn}
L. Bronsard and R. V. Kohn, 
\textit{Motion by mean curvature as the singular limit of Ginzburg-Landau dynamics},
J. Diff. Eqns. {\bf 90} (1991), 211--237.
\bibitem{bcf}
W. K. Burton, N. Cabrera and F. C. Frank, 
\textit{The growth of crystals and the equilibrium structure of their surfaces}, 
Philos. Trans. Roy. Soc. London. Ser. A. {\bf 243} (1951), 299--358.
\bibitem{chen-giga-goto}
Y.-G. Chen, Y. Giga and S. Goto, 
\textit{Uniqueness and existence of viscosity solutions of generalized mean
curvature flow equations}, 
J. Differe. Geom. {\bf 33} (1991), 749--786.
\bibitem{xchen1992}
X. Chen, 
\textit{Generation and propagation of interface in reaction-diffusion equations},
J. Diff. Eqns. {\bf 96} (1992), 116--141.
\bibitem{xchen1996}
X. Chen,
\textit{Global asymptotic limit of solutions of the {C}ahn-{H}illiard equation},
J. Differential Geom. {\bf 44} (1996), 262--311.
\bibitem{evans-soner-souganidis}
L. C. Evans, H. M. Soner and P. E. Souganidis, 
\textit{Phase transitions and generalized motion by mean
curvature}, Comm. Pure Appl. Math. {\bf 45} (1992), 1097--1123.
\bibitem{evans-spluck1991}
L. C. Evans and J. Spruck, 
\textit{Motion of level sets by mean curvature I},
J. Differe. Geom. {\bf 33} (1991), 635--681.
\bibitem{federer} 
H. Federer,
\textit{Geometric Measure Theory},
Springer-Verlag, New York, (1969).
\bibitem{giga2006}
Y. Giga, 
\textit{Surface evolution equations},
Birkh\"{a}user Verlag, Basel (2006).
\bibitem{giusti}
E. Giusti,
\textit{Minimal surfaces and functions of bounded variation},
Birkh\"auser, Boston (1984).
\bibitem{hutchinson}
J. E. Hutchinson,
\textit{Second fundamental form for varifolds and the existence of
surfaces minimising curvature}, Indiana Univ. Math. J. {\bf 35} (1986), 45--71.
\bibitem{Ilmanen}
T. Ilmanen,
\textit{Convergence of the Allen-Cahn equation to Brakke's motion by mean curvature},
J. Differential Geom., {\bf 38} (1993), 417--461.
\bibitem{ilmanen1994}
T. Ilmanen, 
\textit{Elliptic regularization and partial regularity for motion by mean curvature}, 
Mem. Am. Math. Soc. {\bf 108} (1994).
\bibitem{kasai-tonegawa}
K. Kasai, and Y. Tonegawa,
\textit{A general regularity theory for weak mean curvature flow},
Calc. Var. Partial Differential Equations, {\bf 50} (2014), 1--68.
\bibitem{kim-tonegawa}
L. Kim and Y. Tonegawa, 
\textit{On the mean curvature flow of grain boundaries},
Ann. Inst. Fourier (Grenoble) {\bf 67} (2017), 43--142. 
\bibitem{liu-sato-tonegawa}
C. Liu, N. Sato and Y. Tonegawa,
\textit{On the existence of mean curvature flow with transport term},
Interfaces Free Bound., {\bf 12} (2010), 251--277.
\bibitem{liu-sato-tonegawa2}
C. Liu, N. Sato and Y. Tonegawa,
\textit{Two-phase flow problem coupled with mean curvature flow},
Interfaces Free Bound. {\bf 14} (2012), 185--203.
\bibitem{liu-walkington}
C. Liu and N. J. Walkington,
\textit{An Eulerian description of fluids containing visco-elastic particles}
Arch. Ration. Mech. Anal. {\bf 159} (2001), 229--252.
\bibitem{luckhaus-sturzenhecker}
S. Luckhaus and T. Sturzenhecker, 
\textit{Implicit time discretization for the mean curvature flow equation},
Calc. Var. PDE {\bf 3} (1995), 253--271.
\bibitem{meyers-Ziemmer1977}
N. G. Meyers and W. P. Ziemmer,
\textit{Integral inequalities of Poincar\'{e} and Wirtinger type for BV functions},
Amer. J. Math. {\bf 99} (1977), 1345--1360.
\bibitem{modica}
L. Modica,
\textit{A gradient bound and a {L}iouville theorem for nonlinear {P}oisson equations},
Comm. Pure Appl. Math. {\bf 38} (1985), 679--684.
\bibitem{mugnai-roger2008}
L. Mugnai and M. R{\"o}ger,
\textit{The {A}llen-{C}ahn action functional in higher dimensions},
Interfaces Free Bound., {\bf 10} (2008), 45--78.
\bibitem{mugnai-roger2011}
L. Mugnai and M. R{\"o}ger,
\textit{Convergence of perturbed {A}llen-{C}ahn equations to forced mean curvature flow},
Indiana Univ. Math. J., {\bf 60} (2011), 41--75.
\bibitem{roger-schatzle}
M. R{\"o}ger and R. Sch{\"a}tzle,
\textit{On a modified conjecture of {D}e {G}iorgi},
Math. Z., {\bf 254} (2006), 675--714.
\bibitem{rubinstein-sternberg-keller}
J. Rubinstein, P. Sternberg and J. B. Keller, 
\textit{Fast reaction, slow diffusion and curve shortening}, 
SIAM J. Appl. Math. {\bf 49} (1989), 116--133.
\bibitem{simon}
L. Simon, 
\textit{Lectures on geometric measure theory},
Proc. Centre Math. Anal. Austral. Nat. Univ. {\bf 3} (1983).
\bibitem{soner1993}
H. M. Soner,
\textit{Motion of a set by the curvature of its boundary},
J. Differential Equations, {\bf 101} (1993), 313--372.
\bibitem{soner1995}
H. M. Soner, 
\textit{Convergence of the phase-field equations to the Mullins-Sekerka problem with kinetic
undercooling},
Arch. Rational Mech. Anal. {\bf 131} (1995), 139--197.
\bibitem{soner1997}
H. M. Soner, 
\textit{Ginzburg-Landau equation and motion by mean curvature. I. Convergence},
J. Geom. Anal. {\bf 7} (1997), 437--475.
\bibitem{takasao2017}
K. Takasao,
\textit{Existence of weak solution for volume preserving mean curvature flow via phase field method},
Indiana Univ. Math. J., {\bf 66} (2017), 2015--2035.
\bibitem{takasao-tonegawa}
K. Takasao and Y. Tonegawa,
\textit{Existence and regularity of mean curvature flow with transport term in higher dimensions},
Math. Ann. {\bf 364} (2016), 857--935.
\bibitem{tonegawa2014}
Y. Tonegawa,  
\textit{A second derivative H{\"o}lder estimate for weak mean curvature flow}. 
Adv. Calc. Var. {\bf 7} (2014), 91--138.
\bibitem{tonegawa-book}
Y. Tonegawa, 
\textit{Brakke's mean curvature flow},
SpringerBriefs in Mathematics, (2019).
\bibitem{ziemer1989}
W. P. Ziemer,
\textit{Weakly differentiable functions},
Springer-Verlag (1989).
\end{thebibliography}
\end{document}